%% file: ms.tex
\documentclass[a4paper]{amsart}

\usepackage{amssymb}
\usepackage{tikz}
\usepackage{tikz-3dplot}
\usepackage{graphicx}
\usepackage{subcaption}
\usepackage[]{geometry}
\usepackage[]{hyperref}
\usepackage{cleveref}

\theoremstyle{plain}
\newtheorem{theorem}{Theorem}[subsection]
\theoremstyle{definition}
\newtheorem{definition}[theorem]{Definition}
\newtheorem{lemma}[theorem]{Lemma}
\newtheorem{corollary}[theorem]{Corollary}

\theoremstyle{remark}
\newtheorem{remark}[theorem]{Remark}
\newtheorem{example}[theorem]{Example}
\newtheorem{notation}{Notation}[section]

\makeatletter
\def\subsection{\@startsection{subsection}{2}%
  \z@{.5\linespacing\@plus.7\linespacing}{.3\linespacing}%
  {\normalfont\bfseries}}
\def\subsubsection{\@startsection{subsubsection}{2}%
  \z@{.5\linespacing\@plus.7\linespacing}{.3\linespacing}%
  {\normalfont\bfseries}}
\makeatother

\input{paper_commands.tex}

\begin{document}
%%%%%%%%%%%%% TOP MATTER %%%%%%%%%%%%%
\title{Real Tropical Singularities and Bergman Fans}
\author{Christian J\"urgens}
\address{Christian J\"urgens, Fachbereich Mathematik, Eberhard Karls Universit\"at T\"ubingen, 72076 T\"ubingen.}
\email{chju@math.uni-tuebingen.de}
\thanks{The author gratefully acknowledges support by DFG-grant MA 4797/5-1.}

\begin{abstract}	
In this paper, we classify singular real plane tropical curves by means of subdivisions of Newton polytopes. First, we introduce signed Bergman fans (generalizing positive Bergman fans from \cite{ardilaklivanswilliams}) that describe real tropicalizations of real linear spaces  (\cite{tabera}). Then, we establish a duality of real plane tropical curves and signed regular subdivisions of the Newton polytope and explore the combinatorics. We define a signed secondary fan that parametrizes real tropical Laurent polynomials and study the subset providing singular real plane tropical curves. A cone of the signed secondary fan is of maximal dimensional type if its corresponding subdivision contains only marked points (\cite{markwigmarkwigshustin}). These cones parametrize real plane tropical curves. We classify singular real plane tropical curves of maximal dimensional type.
\end{abstract}
\maketitle
%%%%%%%%%%%%% MAIN MATTER %%%%%%%%%%%%%
\input{sec_1}	% real tropical geometry
\input{sec_2}	% oriented matroids & signed bergman fans
\input{sec_3}	% real sing. trop. curves
\bibliographystyle{amsalpha}
\bibliography{bib}
\end{document}

%% file: paper_commands.tex
\DeclareMathOperator{\Proj}{\mathbb{P}}				% Proj
\DeclareMathOperator{\RealTropicalgroup}{\mathbb{TR}}

\DeclareMathOperator{\N}{\mathbb{N}}
\DeclareMathOperator{\Z}{\mathbb{Z}}
\DeclareMathOperator{\Q}{\mathbb{Q}}
\DeclareMathOperator{\R}{\mathbb{R}}
\DeclareMathOperator{\C}{\mathbb{C}}
\newcommand{\Field}{\mathbb{K}}					% ground field
\newcommand{\ComplexField}{\Field_{\C}}				% complex 
\newcommand{\RealField}{\Field_{\R}}				% real Puiseux
\newcommand{\ComplexPuiseux}{\C\{\{t\}\}}				% complex Puiseux series
\newcommand{\RealPuiseux}{\R\{\{t\}\}}					% real Puiseux series
\newcommand{\Units}[1]{{#1}^*}

\newcommand{\RealLauring}[1]{\RealField\left[x_1^\pm,\ldots,x_{#1}^\pm\right]}	% Real Laurent polynomial ring for torus T^n
\newcommand{\Realcoefficientring}[1]{\RealField\left[y_1,\ldots,y_{#1}\right]}

\newcommand{\RealGenericringnvars}{\RealField\left[y_1,\ldots,y_m\right]\left[x_{1}^\pm,\ldots,x_{n}^\pm\right]}
\newcommand{\RealTropaffring}[1]{\RealTropicalgroup\left[w_1,\ldots,w_{#1}\right]}% Tropical polynomial ring for affine space R^n
\newcommand{\RealTroplauring}[1]{\RealTropicalgroup\left[w_1^\pm,\ldots,w_{#1}^\pm\right]}% Tropical polynomial ring for affine..
\newcommand{\RealTroplauringtwovars}{\RealTropicalgroup\left[w_1^\pm,w_2^\pm\right]}% Tropical polynomial ring for plane
\newcommand{\RealTorus}[1]{T_{\R}^{#1}}					% real Torus
\newcommand{\ComplexTorus}[1]{T_{\C}^{{#1}}}
\newcommand{\Mod}[1]{\big/\ones{{#1}}}					% R^n+1 / 1R

\newcommand{\Dual}[1]{\left( {#1} \right)^\vee}				% dual space

\newcommand{\groundset}[1]{\left[{#1}\right]}			% ground set [n]
\newcommand{\circuits}{\mathcal{C}}				% circuits matroid
\newcommand{\affinespan}[1]{\mathrm{aff}\left({#1}\right)}	% span
\newcommand{\biglattice}{\flag_\text{big}}			% big lattice
\newcommand{\initialcircuit}[2]{\mathrm{in}_{{#1}}\left( {#2} \right)}		% initial circuit
\newcommand{\initialcircuits}[1]{\mathrm{in}_{{#1}}\left( \circuits \right)}	% intial circuits of M_w
\newcommand{\reoriented}[2]{{}_{-{#1}}{#2}}			% reoriented circuit
\newcommand{\reorientedmatroid}[2]{{}_{-{#1}}{#2}}		% reoriented matroid
\newcommand{\reorientedcircuits}[1]{{}_{-{#1}}\circuits}	% reoriented circuits
\newcommand{\umatroid}[1]{\underline{{#1}}}			% underlying matroid
\newcommand{\svector}{\mathcal{S}}				% sign vectors
\newcommand{\pvector}[1]{\{\pm\}^{{#1}}}
\newcommand{\cvector}[1]{\mathcal{L}_{{#1}}}			% covectors
\newcommand{\topes}{\mathcal{Y}}				% topes
\newcommand{\signedbergmanfan}[2]{\mathcal{B}^{{#1}}\left({#2}\right)}

\newcommand{\idealmatroid}[1]{M\left({#1}\right)}		% matroid associated to I
\newcommand{\vectormatroid}[1]{M\left[{#1}\right]}		% vector matroid

\newcommand{\support}{\mathcal{A}}				% support 
\newcommand{\ones}[1]{\mathbf{1}_{{#1}}}			% ones = (1,...,1)
\newcommand{\zeros}[1]{\mathbf{0}_{{#1}}}			% zeros = (0,...,0)
\newcommand{\Realsingsat}[1]{\nabla_{\R,{#1}}}			% family of real sing. plane trop. curves
\newcommand{\flag}{\mathcal{F}}					% flag (of subsets)
\newcommand{\tdim}[1]{\mathrm{tdim}({#1})}			% type-dimension
\newcommand{\rowspace}{\mathrm{rowspace}}

\newcommand{\secondaryfan}[1]{\mathrm{Sec}_{{#1}}}		% secondary fan
\newcommand{\signedsecondaryfan}[1]{\{\pm\}^m\times\secondaryfan{#1}}
\DeclareMathOperator{\conv}{\mathrm{conv}}			% convex hull
\DeclareMathOperator{\relint}{\mathrm{relint}}			% relative interior
\newcommand{\RA}{\R^{|\support|}}				% R^A support 
\newcommand{\sk}[2]{\langle {#1},{#2}\rangle}			% skalar product

\newcommand{\initialform}[2]{\mathrm{in}_{{#1}}\left({#2}\right)}	% intial form

\newcommand{\Var}[1]{\mathcal{V}\left( {#1} \right)}		% Variety
\newcommand{\Th}[1]{\mathcal{T}\left( {#1} \right)}		% tropical hypersurface
\newcommand{\RealTh}[1]{\mathcal{T}_{\R}\left( {#1} \right)}	% real tropical hypersurface
\newcommand{\Ideal}[1]{\mathcal{{#1}}}				% Ideal
\newcommand{\Newt}[1]{\mathrm{Newt}\left( {#1} \right)}		% Newton polytope
\DeclareMathOperator{\Kleingroup}{\Z_2\times\Z_2}
\DeclareMathOperator{\Image}{\mathrm{Im}}			% Image of a map

\newcommand{\sign}{\mathrm{sign}}				% sign function / signum
\newcommand{\val}{\mathrm{val}}					% valuation
\newcommand{\complextrop}{\mathrm{trop}_{\C}}				% trop
\newcommand{\realtrop}{\mathrm{trop}_{\R}}			% real trop
\newcommand{\If}{~\text{if}~}					% If Textbaustein
\newcommand{\Then}{~\text{then}~}				% Then Textbaustein
\newcommand{\Or}{~\text{or}~}

\renewcommand{\And}{~\text{and}~}

\newcommand{\With}{~\text{with}~}

\newcommand{\Acyclic}{\text{-acyclic}}

%% file: sec_1.tex
% !TEX root = /Users/christianjurgens/Dropbox/UniSL/Projekte/Real singular plane tropical curves/Paper/ms.tex

%%%%%%%%%%%%% INTRODUCTION %%%%%%%%%%%%%

Singularities play a decisive role in algebraic geometry. In recent years, a lot of effort was put into the translation of the concept of a singularity to tropical geometry (e.g. \cite{dickensteintabera}, \cite{markwigmarkwigshustin}, \cite{markwigmarkwigshustin2}, \cite{dickensteinfeichtnersturmfels}). However, this process is far from complete. To the present day, there is no intrinsic definition of a singular tropical variety and current research is limited to tropical hypersurfaces. A common approach is to call a tropical hypersurface singular if there is a singular algebraic hypersurface tropicalizing to it. A popular instrument for the investigation of singular tropical hypersurfaces is the $A$-discriminant defined by Gelfand, Kapranov and Zelevinsky (\cite{dickensteinfeichtnersturmfels}, \cite{gelfandkapranovzelevinsky}). Its zero set parametrizes the family of Laurent polynomials with fixed support (given by the columns of the matrix $A\in\Z^{n\times n}$) that provide singular hypersurfaces.
\smallskip

We consider singular tropical hypersurfaces for real tropical geometry viewed as real tropicalizations of geometry over real Puiseux series. Tropical geometry is a powerful tool for the study of real geometry. The famous Viro's patchworking method can be used to construct real varieties with prescribed topology building on tropical varieties(\cite{MR2465564}, \cite{MR1905317}, \cite{MR1413249}, \cite{MR1036837}, \cite{MR1745011}, \cite{MR3574512}, \cite{MR3044488}). Tropical curves can be used to study counts of real curves satisfying incidence conditions (\cite{MR2137980}, \cite{MR2400524}, \cite{MR2466076}). Real tropicalization also plays a role in the study of convex geometry and tropical polytopes, with applications e.g. in complexity theory (\cite{MR2367318}, \cite{MR3336300}, \cite{MR3504692}). The chart of a real tropicalization with all signs positive (also called the positive part of a tropical variety) plays a role in the study of total positivity, which classically concerns matrices with all minors being positive and which nowadays has connections to the study of Cluster algebras (\cite{MR2164397}, \cite{MR1648700}, \cite{MR2525057}, \cite{MR2813307}, \cite{MR3279534}, \cite{MR3607668}). Tropical geometry is also a tool to study solutions of systems of polynomial equations over the reals (\cite{MR2555958}, \cite{MR2795209}, \cite{MR2783101}, \cite{2017arXiv170302272E}).\smallskip

Here, we view real tropicalization as tropicalization with signs (\cite{tabera}). Due to the signs, a real plane tropical curve can be described by sets of charts (cf. \Cref{defi:charts}). These charts are dual to signed regular marked subdivisions (cf. \Cref{defi:signedmarkedsubdivision}). We define a signed secondary fan whose equivalence classes correspond to sets of charts defining a real plane tropical curve ``up to sign permutation'' (cf. \Cref{defi:signedsecondaryfan}). Our main result is a characterization of real plane tropical curves of maximal dimensional type with a singularity in a fixed point stated in \Cref{subsec:firststepsrealclassification}. Our methods can also be applied to the case of real tropical surfaces with a singularity in a fixed point. We refer to \cite[Theorem 4.3.3.9]{juergens} for details.\smallskip

\Cref{sec:realtropicalgeometry} contains preliminaries concerning real tropicalization and real tropical curves as signed versions in charts following \cite{tabera}. \Cref{sec:signedbergmanfans} reviews oriented matroids and generalizes the associated positive Bergman fan from \cite{ardilaklivanswilliams} for arbitrarily signed Bergman fans. Signed Bergman fans are real tropicalizations of real linear spaces. \Cref{sec:realsingularplanetropicalcurves} exploits the real tropicalization of the linear space that parametrizes the family of real curves with a singularity in a fixed point  to characterize the combinatorics of singular real plane tropical curves. This is based on \cite{markwigmarkwigshustin} which provides a combinatorial characterization in the non-real (respectively unsigned) case. Our main result is \Cref{theo:classificationcurvesmaximaldimensionaltype} which classifies singular real plane tropical curves.
\section*{Acknowledgments}
I would like to offer special thanks to Hannah Markwig for all the helpful comments, discussions and the careful proofreading.

\section{Real Tropical Geometry}\label{sec:realtropicalgeometry}

The following preliminary section is based on \cite{tabera}.

%%%%%%%%%%
\subsection{Real Puiseux Series and the Real Tropical Group}\label{subsec:realpuiseuxseries}

We work over the non-archimedian valued field of real Puiseux series
$\RealPuiseux$. An element $a\in\RealPuiseux$ is a power series in the indeterminate $t$ with rational exponents that have a common denominator and the coefficients are real valued, i.e. there exist $i\in\Z$, $n\in\N$ and $a_k\in\R$ for all $k\geq i$ such that
\begin{align}\label{eq:realPuiseux}
a = \sum_{k = i}^\infty a_k t^{\frac{k}{n}}\in\RealPuiseux.
\end{align}
The valuation $\val:\RealPuiseux^*\rightarrow\R$ of $0\neq a\in\RealPuiseux$ is the lowest exponent of $t$ appearing in $a$. Note that $\RealPuiseux$ is not algebraically closed. However, $\RealPuiseux$ is an ordered field via $a>0$ if and only if the coefficient of $t^{\val(a)}$ is positive. We write $\RealField=\RealPuiseux$ and denote the $n$-dimensional torus by $\RealTorus{n}$.

\begin{definition}[Sign function]\label{defi:signfunction}
Let $a\in\Units{\RealField}$ be a real Puiseux series. We define a \textit{sign function}
\begin{equation}\notag
s: \RealField^* \longrightarrow \left\{-1,1\right\},\quad s(a) = \begin{cases}
       1~\text{if}~a>0,\\
       -1~\text{if}~a<0.
       \end{cases} 
\end{equation}
\end{definition}

In the following, we work over $\RealField$. If we work over $\ComplexField=\ComplexPuiseux$ we refer to it as \textit{the complex case}. We denote the $n$-dimensional torus over $\ComplexField$ by $\ComplexTorus{n}$. The underlying algebraic structure of our real tropical objects is defined as follows:

\begin{definition}[Real tropical group]\label{defi:realtropicalgroup}
The tuple $\RealTropicalgroup =\left( \{-1,1\}\times \R,\odot_{\R}\right)$ is called \textit{real tropical group}. The composition of $(a,x),(b,y)\in\RealTropicalgroup$ via $\odot_{\R}$ is defined by
\begin{equation}\notag
(a,x)\odot_{\R} (b,y) = (ab,x+y).
\end{equation}
\end{definition}

There is no reasonable ``addition'' in $\RealTropicalgroup$ as it is not clear how to ``add'' signs.
%%%%%%%%%%
\subsection{Real Tropicalization}\label{subsec:realtropicalization}

First, we fix the notion of a \textit{real tropicalization}:

\begin{definition}[Real tropicalization]\label{defi:realtropicalization}
The \textit{real tropicalization} is the map defined by
\begin{align*}
\realtrop: \Units{\RealField} &\longrightarrow \RealTropicalgroup = \{-1,1\}\times \R\\
x &\longmapsto \left( s(x),-\val(x) \right)
\end{align*}
where $s$ is the sign function (cf. \Cref{defi:signfunction}) and $\val$ the valuation of $\RealField$. For $y=(a,b)\in\RealTropicalgroup$ we call $a$ the sign of $y$ and $b$ the modulus of $y$. By abuse of notation, we refer to the first component of $y$ by $s(y)$ and to the second component of $y$ by $|y|$.
\end{definition}

In comparison to the tropicalization over $\ComplexField$ (defined by $\complextrop=-\val$ where $\val$ is equivalently defined on $\ComplexField$), the real tropicalization takes signs into account.

% Sign vector
\begin{definition}[Sign vector]\label{defi:signvector}
An element $s\in\svector:=\left\{+,-,0\right\}^n\cong\left\{1,-1,0\right\}^n$ is called \textit{sign vector}. If $s\in\pvector{n}$ we call $s$ \textit{pure sign vector}.
\end{definition}

\begin{definition}[Charts]\label{defi:chartsambientspace}
% old: By $\R_v^n$ we denote the copy of $\R^n$ containing the elements $w\in\RealTropicalgroup^n$ with $s(w)=v$ and call it \textit{$v$-chart} of $\RealTropicalgroup^n$.
Let $s\in\svector$ be a pure sign vector. The copy of $\R^{n}$ containing all points of $\RealTropicalgroup^{n}$ with signs $s$, $s\times\R^{n}\subset\RealTropicalgroup$, is called \textit{$s$-chart} of $\RealTropicalgroup^{n}$.
\end{definition}

\begin{definition}[Real tropicalized variety]\label{defi:realtropicalizedvariety}
Let $X=\Var{I}\subset\RealTorus{n}$ be an algebraic variety defined by $I\subset\RealLauring{n}$. The \textit{real tropicalization} of $X$ is the closure of the set $\left\{\realtrop(x):x\in X\right\}$ in $\RealTropicalgroup^n$ and is denoted by $\realtrop(X)$.
\end{definition}

\begin{definition}[Real tropical polynomial]\label{defi:realtropicalpolynomial}
Let $\support=\{\alpha_1,\ldots,\alpha_m\}\subset\Z^n$ be a finite set. A \textit{real tropical polynomial} $f\in\RealTroplauring{n}$ with support $\support$ is a formal sum of products of real tropical monomials with respect to $\support$ and coefficients in $\RealTropicalgroup$, i.e.
\begin{equation}\notag
f = \bigoplus_{i} p_i w^{\alpha_i}.
\end{equation}
A real tropical polynomial $f$ yields a piecewise affine linear function by restricting to the modulus of $f$, i.e. $|f|:\RealTropicalgroup^n \longrightarrow \R$ is the map defined by $|f|(w)=\max_{i}\left\{|p_i|+\sk{\alpha_i}{|w|}\right\}$.
\end{definition}

As in the complex case, we can translate real Laurent polynomials to real tropical Laurent polynomials:

\begin{definition}[Real tropicalized polynomials]\label{defi:realtropicalizedpolynomials}
Let $\support=\{\alpha_1,\ldots,\alpha_m\}\subset\Z^n$ be the support of a real Laurent polynomial $F=\sum_{i} a_i x^{\alpha_i}\in\RealLauring{n}$. The \textit{real tropicalization of $F$} is the real tropical Laurent polynomial defined by \[f=\realtrop(F) := \bigoplus_{i} \left( s(a_i),-\val\left( a_i \right) \right) w^{\alpha_i}\in\RealTroplauring{n}.\]
\end{definition}

\begin{definition}[Real tropical hypersurface]\label{defi:realtropicalhypersurface}
Let $f\in\RealTroplauring{n}$ be a real tropical Laurent polynomial with support $\support=\{\alpha_1,\ldots,\alpha_m\}\subset\Z^n$. The \textit{real tropical hypersurface} $\RealTh{f}$ is the set of all $y=\left( y_1,\ldots,y_n \right)\in\RealTropicalgroup^{n}$ for which exist $i\neq j$ such that $|f|(w)$ attains its maximum at $\alpha_i$ and $\alpha_j$, i.e. \[|p_i| + \sk{\alpha_i}{|y|} = |p_j| + \sk{\alpha_j}{|y|} \geq |p_k| + \sk{\alpha_k}{|y|}\quad \forall k\neq i,j,\]
and the signs at $\alpha_i$ and $\alpha_j$ are opposite, i.e. \[s\left( p_i \right) \prod_{l=1}^n s\left( y_l \right)^{(\alpha_i)_l} \neq s\left( p_j \right) \prod_{l=1}^n s\left( y_l \right)^{(\alpha_j)_l}.\]
\end{definition}

\begin{example}[{\cite[Example 3.2]{tabera}}]\label{exam:nokapranov}
Consider $F=x^2-x+1\in\RealField\left[x\right]$. Then, $\Var{F}=\emptyset\subset\RealTorus{}$ but $\RealTh{\realtrop\left( F \right)}=\{\left( +,0 \right)\}$, i.e. $\realtrop\left( \Var{F} \right) \subsetneq \RealTh{\realtrop\left( F \right)}$.
\end{example}

\Cref{exam:nokapranov} show that there is no analog to Kapranov's Theorem.

\begin{definition}[Real tropical basis]\label{defi:realtropicalbasis}
Let $I\subset\RealLauring{n}$ be an ideal. A \textit{real tropical basis of} $I$ is a finite set $\left\{F_1,\ldots,F_k\right\}$ of $I$ such that 
\begin{equation}\notag
\realtrop\left(\Var{I}\right) = \bigcap_{i=1}^k \RealTh{\realtrop(F_i)}.
\end{equation}
\end{definition}

In \cite{tabera}, tropical bases for certain classes of ideals were studied. We explore the special case of real tropicalizations of linear subspaces $V=\Var{\Ideal{I}}\subset\RealTorus{n}$, i.e. let $\Ideal{I}=\langle l_1,\ldots,l_{n-k}\rangle\subset\RealLauring{n}$ be an ideal generated by linear forms $l_i=\sum_{j} a_{ij} x_j$. If $\Ideal{I}$ contains a generating set of polynomials whose coefficients have trivial valuation we say that  $\Ideal{I}$ has \textit{constant coefficients}.

\begin{remark}[Oriented matroids from linear ideals]\label{rema:orientedmatroidassociatedtoideal}
Let $\{l_1,\ldots,l_{n-k}\}$ be a generating set of the linear ideal $\Ideal{I}\subset\RealLauring{n}$. The linear form $l_{i}$ can be written as
\begin{equation}
l_i = \sum_{j\in J_1} a_{ij} x_j - \sum_{j\in J_2} a_{ij} x_j\label{eq:signedlinearform}
\end{equation}
where $J_1,J_2$ are disjoint sets and $a_{ij}>0$ for all $i,j$. We work in the constant coefficient case, i.e. $\val(a_j)=\val(b_j)=0$. The linear form, written as in \Cref{eq:signedlinearform}, provides a signed circuit $C_{i}$ by defining $C_{i}^+=J_1$ and $C_{i}^- = J_2$. All signed circuits obtained this way form an oriented matroid denoted by $M=\idealmatroid{\Ideal{I}}$. Similar to the complex case, we call $\idealmatroid{\Ideal{I}}$ the \textit{matroid associated to $\Ideal{I}$} (cf. \cite{sturmfels}). The real tropicalization of $l_{i}$ equals
\begin{equation}
\realtrop\left( l_{i} \right)=\left( \bigoplus_{j\in J_1} 0^+ w_j \right) \oplus \left( \bigoplus_{j\in J_2} 0^- w_j \right).\label{eq:realtropicallinearform}
\end{equation}
Note that we can determine $\realtrop(l_{i})$ uniquely from $C_{i}$. Thus, we refer to the tropical linear form obtained from $C$ by $l_C$.
\end{remark}

In the complex case, the underlying matroid $\umatroid{M}$ of $M$ determines the tropicalization of the linear space $\Var{\Ideal{I}}\subset\ComplexTorus{n}$. In the real case, we have an analogous statement for the oriented matroid $M$ with regard to $\realtrop\left( \Var{\Ideal{I}} \right)$. More precisely, the signed circuits form a tropical basis for $\Ideal{I}$.
% % THEO: positive Bergman fan
\begin{theorem}[{\cite[Theorem 3.14]{tabera}}]\label{theo:realtropicalbasis}
Let $\Ideal{I}\subset\RealLauring{n}$ be a linear ideal with constant coefficients and $M$ the associated oriented matroid with signed circuits $\circuits$. Then:\[\realtrop\left(\Var{\Ideal{I}}\right)=\bigcap_{C\in\circuits} \RealTh{l_C}.\]
\end{theorem}
This theorem also holds in the non-constant coefficient case.

%% file: sec_2.tex
% !TEX root = /Users/christianjurgens/Dropbox/UniSL/Projekte/Real singular plane tropical curves/Paper/ms.tex

%%%%%%%%%%%%% signed Bergman FANS %%%%%%%%%%%%%

\section{Oriented Matroids and Signed Bergman Fans}\label{sec:signedbergmanfans}

Oriented matroids can be used to study vector configurations over the reals. Their combinatorics and relation to (real) tropical geometry are studied (e.g. in \cite{MR2764796}, \cite{MR2511751}, \cite{MR3612868}, \cite{2017arXiv170801329G}).\smallskip

In the following, we write $\groundset{m}=\{1,\ldots,m\}$ and, if not differently specified, $E=\groundset{m}$ denotes a set for some $m\in\N$. Let $\Ideal{I}\subset\RealLauring{n}$ be a linear ideal and $M$ the associated oriented matroid. Then, the all-positive chart $\realtrop(\Var{\Ideal{I}})\cap(+)^n\times\R^n$ was called the positive Bergman fan of the oriented matroid M. In this section we adapt the results of \cite{ardilaklivanswilliams} for arbitrary sign vectors and other charts of $\realtrop(\Var{\Ideal{I}})$.\smallskip

\subsection{Oriented Matroids}\label{subsec:orientedmatroids}

In order to define oriented matroids we need some terminology concerning \textit{signed sets}:

\begin{definition}[Signed sets]\label{defi:signedsets}
A \textit{signed set} $X$ of $E$ is a subset $\underline{X}\subseteq E$ with a partition $\left( X^+,X^- \right)$ of $\underline{X}$ where $X^+$ is the set of \textit{positive elements} of $X$ and $X^-$ is the set of \textit{negative elements} of $X$. Hence, $\underline{X}=X^+\cup X^-$ is the \textit{support} of $X$ and $|X|$ denotes the cardinality of $\underline{X}$.
\end{definition}

Additionally, we define $X^0 =E\setminus\underline{X}$. We notice that if $X$ is a signed set then $-X$ is also a signed set via $(-X)^+=X^-$ and $(-X)^-=X^+$. By convention, we write $i$ for $i\in X^+$ and $\overline{i}$ for $i\in X^-$. Now, we define an \textit{oriented matroid}:

\begin{definition}[Oriented matroid]\label{defi:orientedmatroid}
An ordered pair $\left( E,\circuits \right)$ consisting of a ground set $E$ and a collection $\circuits$ of signed sets $C$ of $E$ is an \textit{oriented matroid} $M$ if and only if the following conditions are satisfied:
\begin{itemize}
\item[(C0)] $\emptyset\notin\circuits$
\item[(C1)] $\circuits = -\circuits$
\item[(C2)] $\forall X,Y\in\circuits:\If\underline{X}\subseteq\underline{Y},\Then, X=Y\Or X=-Y$
\item[(C3)] $\forall X,Y\in\circuits$ with $X\neq -Y$ and $e\in X^+\cap Y^-$ there is a $Z\in\circuits$ such that
\renewcommand{\labelitemii}{$\bullet$}
\begin{itemize}
\item $Z^+\subseteq\left( X^+\cup Y^+ \right)\setminus\{e\}$                                                                                                           \item $Z^-\subseteq\left( X^-\cup Y^- \right)\setminus\{e\}$                                                                                                                 \end{itemize}
\end{itemize}
\end{definition}

Conditions (C0) to (C3) are called \textit{circuit axioms} of oriented matroids and the elements of $\circuits$ are called \textit{signed circuits}. In the remaining part of this section, let $M$ denote an oriented matroid on the ground set $E=\groundset{n}$ of rank $k$. If nothing else is mentioned we consider all matroids (and its properties) with signs, i.e. we do not always write ``oriented'' explicitly.\smallskip

\begin{remark}[Link to matroids]\label{rema:underlyingmatroid}
If we forget about signs the circuit axioms (C0) - (C3) reduce to the circuit axioms in classical matroid theory (\cite[Chapter 1, Section 1.1]{oxley}). The collection of \textit{signed circuit supports} $\underline{\circuits}=\{\underline{C}:C\in\circuits\}$ forms a collection of circuits of a matroid $\umatroid{M}$ called the \textit{underlying matroid of} $M$. Hence, an oriented matroid $M$ inherits properties of its underlying matroid $\umatroid{M}$ (e.g. its rank).
\end{remark}

% DEFI: Reorientation of matroids
\begin{definition}[Reorientation]\label{defi:reorientation}
Let $M=\left( E,\circuits \right)$ be an oriented matroid. For $A\subseteq E$ we define a \textit{reorientation of $M$} with respect to $A$ as follows: for each signed circuit $C\in\circuits$ we define the reoriented signed circuit $\reoriented{A}{C}$ by $\left(\reoriented{A}{C}\right)^+=\left( C^+\setminus A \right)\cup \left( C^-\cap A \right)$ and $\left(\reoriented{A}{C}\right)^-=\left( C^-\setminus A \right)\cup \left( C^+\cap A \right)$. By $\reorientedcircuits{A}$ we denote the set of reoriented signed circuits. The oriented matroid given by $\reorientedcircuits{A}$ is denoted by $\reorientedmatroid{A}{M}$.
\end{definition}

\begin{remark}[Oriented matroids of point configurations]\label{rema:orientedvectormatroids}
Let $\support=\{\alpha_1,\ldots,\alpha_m\}\subset\R^n$ be a point configuration and $A\in\R^{n\times m}$ the representing matrix. An element $\lambda\in\ker(A)$ provides a set $\underline{C}=\{i\in\groundset{m}:\lambda_{i}\neq 0\}$. If $\underline{C}$ is inclusion-minimal then $\lambda$ provides a minimal linear dependence among the columns of $A$ and we call $\underline{C}$ circuit. The partition of $\underline{C}$ into subsets $C^\pm:=\{i\in\groundset{m}:\lambda_i \gtrless 0\}$ provides a signed circuit $C$ whose support is $\underline{C}$. The collection of signed circuits (also denoted by $\circuits$) forms an oriented matroid denoted by $M[A]=\left(\groundset{m},\circuits \right)$ called \textit{oriented vector matroid} (\cite[Theorem 3.2.4]{bjoernerlasvergnassturmfelswhiteziegler}). 
\end{remark}

\begin{example}\label{exam:orientedvectormatroid}
Consider the point configuration $\support=\{\alpha_1,\ldots,\alpha_5\}\subset\R^2$ illustrated by the columns of the matrix
\begin{equation}\notag
A=\begin{pmatrix}
  1 & 0 & 1 & -1 & -2\\
  0 & 1 & 1 & -1 & -1
  \end{pmatrix}.
\end{equation}
The point configuration is shown in \Cref{fig:orientedvectormatroid} (A). The signed circuits of $M[A]$ are \[\{12\overline{3},124,125,34,135,1\overline{4}5,\overline{2}35,\overline{2}\overline{4}5\}.\]
\end{example}

\begin{figure}[ht]
\centering
\def\x{7/3}
\def\y{7/3}
\subcaptionbox{Point configuration $\support\subset\R^2$.}[0.49\linewidth]{
\begin{tikzpicture}
% origin
\coordinate (B0) at (0,0);
% axes
\draw[] (0,-\y) -- (0,\y);
\draw[] (-\x,0) -- (\x,0);
\draw[color=gray,thin] (-\x,-1) -- (\x,-1);
\draw[color=gray,thin] (-\x,-2) -- (\x,-2);
\draw[color=gray,thin] (-\x,1) -- (\x,1);
\draw[color=gray,thin] (-\x,2) -- (\x,2);
\draw[color=gray,thin] (-1,-\y) -- (-1,\y);
\draw[color=gray,thin] (-2,-\y) -- (-2,\y);
\draw[color=gray,thin] (1,-\y) -- (1,\y);
\draw[color=gray,thin] (2,-\y) -- (2,\y);
\node at (\x+0.5,0) {$x$};
\node at (0,\y+0.5) {$y$};
% letters
\coordinate[shape=circle,inner sep=2pt,fill=black] (P1) at (1,0);
\coordinate[shape=circle,inner sep=1pt,label={$x_1$}] (L1) at (0.8,0);
\coordinate[shape=circle,inner sep=2pt,fill=black] (P2) at (0,1);
\coordinate[shape=circle,inner sep=1pt,label={$x_2$}] (L2) at (-0.2,1);
\coordinate[shape=circle,inner sep=2pt,fill=black] (P3) at (1,1);
\coordinate[shape=circle,inner sep=1pt,label={$x_3$}] (L3) at (0.8,1);
\coordinate[shape=circle,inner sep=2pt,fill=black] (P4) at (-1,-1);
\coordinate[shape=circle,inner sep=1pt,label={$x_4$}] (L4) at (-1.2,-1);
\coordinate[shape=circle,inner sep=2pt,fill=black] (P5) at (-2,-1);
\coordinate[shape=circle,inner sep=1pt,label={$x_5$}] (L1) at (-2.2,-1);
\path (0,-3) node[below,opacity=0] {V};
\end{tikzpicture}}
\subcaptionbox{Topes and covectors of $M[A]$.}[0.49\linewidth]{
\def\x{3}
\def\y{3}
\begin{tikzpicture}
% origin
\coordinate[shape=circle,inner sep=1pt,fill=black] (B0) at (0,0);
% axes
% \draw[] (0,-\y) -- (0,\y);
% \draw[] (-\x,0) -- (\x,0);
% letters
\coordinate[shape=circle,inner sep=2pt,fill=black] (P1) at (1,0);
\coordinate[shape=circle,inner sep=2pt,fill=black] (P5) at (-2,-1);
\coordinate[shape=circle,inner sep=2pt,fill=black] (P2) at (0,1);
\coordinate[shape=circle,inner sep=2pt,fill=black] (P3) at (1,1);
\coordinate[shape=circle,inner sep=2pt,fill=black] (P4) at (-1,-1);
\draw[color=gray] (0,\y) -- (0,-\y);
\draw[color=gray] (-\x,0) -- (\x,0);
\draw[color=gray] (-\x,\y) -- (\x,-\y);
\draw[color=gray] (-\x,\y) -- (\x,-\y);
\draw[color=gray] (-1.5,\y) -- (1.5,-\y);
\draw[dashed,color=red] (P1) -- (B0);
\draw[dashed,color=red] (P2) -- (B0);
\draw[dashed,color=red] (P3) -- (B0);
\draw[dashed,color=red] (P4) -- (B0);
\draw[dashed,color=red] (P5) -- (B0);

% topes
\coordinate[shape=circle,inner sep=1pt,label=\tiny{$\begin{pmatrix} +\\-\\+\\-\\-\end{pmatrix}$}] (C1) at (2,-1.7);
\coordinate[shape=circle,inner sep=1pt,label=\tiny{$\begin{pmatrix} +\\-\\-\\+\\-\end{pmatrix}$}] (C2) at (1.9,-3.5);
\coordinate[shape=circle,inner sep=1pt,label=\tiny{$\begin{pmatrix} +\\-\\-\\+\\+\end{pmatrix}$}] (C3) at (0.5,-3.5);
\coordinate[shape=circle,inner sep=1pt,label=\tiny{$\begin{pmatrix} -\\-\\-\\+\\+\end{pmatrix}$}] (C4) at (-1.8,-3.0);
\coordinate[shape=circle,inner sep=1pt,label=\tiny{$\begin{pmatrix} -\\+\\-\\+\\+\end{pmatrix}$}] (C5) at (-2,0);
\coordinate[shape=circle,inner sep=1pt,label=\tiny{$\begin{pmatrix} -\\+\\+\\-\\+\end{pmatrix}$}] (C6) at (-1.9,2);
\coordinate[shape=circle,inner sep=1pt,label=\tiny{$\begin{pmatrix} -\\+\\+\\-\\-\end{pmatrix}$}] (C7) at (-0.5,2);
\coordinate[shape=circle,inner sep=1pt,label=\tiny{$\begin{pmatrix} +\\+\\+\\-\\-\end{pmatrix}$}] (C8) at (1.8,1);
\end{tikzpicture}}
\caption{Point configuration and topes of \Cref{exam:orientedvectormatroid}.}
\label{fig:orientedvectormatroid}
\end{figure}
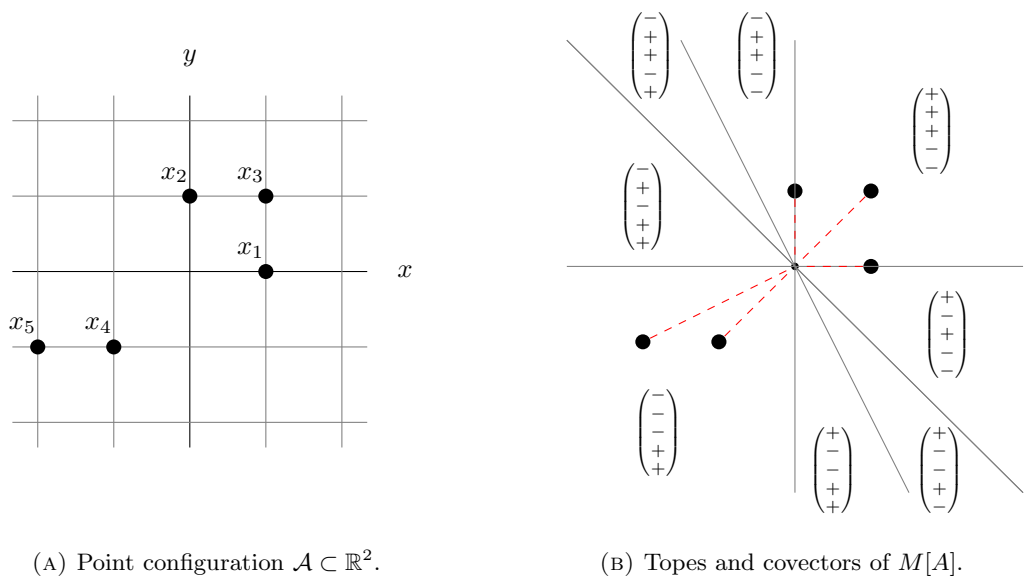

Another approach to oriented matroids is provided by \textit{covectors}. Even though the concept of covectors allows to define an oriented matroid with an axiomatic system equivalent to \Cref{defi:orientedmatroid} (cf. \cite[Chapter 4, \S 1]{bjoernerlasvergnassturmfelswhiteziegler}) we assume that, from now on, oriented matroids always arise from point configurations $\support\subset\R^n$, i.e. we consider oriented vector matroids.

\begin{remark}[Sign vectors and signed sets]\label{rema:signvectorsandsignedsets}
We identify signed sets $X$ of $E$ and sign vectors $s\in\svector$ via \[s_e = \pm \quad\Leftrightarrow\quad e\in X^{\pm}\quad\And{}\quad s_{e}=0\quad\Leftrightarrow\quad e\in X^{0}.\]
We denote the associated elements to $X$ and $s$ by $s(X)$ and $X(s)$ respectively. Via this identification we define $s^+=X^+$ and $s^-,s^0$ analogously.
% Let $V\subseteq\R^n$ be a vector space. The \index{signed support}\textit{signed support} $s(v)$ of a vector $v\in V$ is the signed set defined by $s(v)^\pm=\{i\in\groundset{n}:v_i \gtrless 0\}$. Let $\Supp{V}$ be the set of signed supports of all elements of $v\in V$, $\Circ{V}\subset\Supp{V}$ the subset of inclusion-minimal signed supports.
\end{remark}

\begin{remark}\label{rema:signoperations}
For $a,b\in\svector$ two sign vectors let $a\cdot b$ denote the sign vector obtained by multiplying $a$ and $b$ componentwise, i.e. $\left( a\cdot b \right)_i = a_ib_i$. We define a partial order on $\svector$ by \[a\subseteq b\qquad:\Longleftrightarrow\qquad a^+ \subseteq b^+~\text{and}~a^-\subseteq b^-.\]
\end{remark}

%Covector
\begin{definition}[Covector]\label{defi:covector}
Let $\support=\{\alpha_{1},\ldots,\alpha_{m}\}\subset\R^n$ be a point configuration and $M[A]$ the oriented vector matroid associated to the matrix representation $A\in\R^{n\times m}$ of $\support$. A sign vector $s\in\svector$ is called \textit{covector} of $M[A]$ if there is an element $y\in\Dual{\R^n}$ such that
\begin{equation}\notag
s = \big( \sign\left( y(\alpha_1) \right),\ldots,\sign\left( y(\alpha_m) \right) \big).
\end{equation}
The set of covectors of $M[A]$ is denoted by $\cvector{M[A]}\subseteq\svector$.
\end{definition}

% Topes
\begin{remark}[Topes]\label{rema:topes}
If $v\in\cvector{M}$ is pure then $v$ is called \textit{tope}. We denote the set of topes by $\topes$. Thus, $\Dual{\R^n}$ is subdivided by the hyperplanes defined by $\alpha\in\support$ into cells where each full dimensional cell is indexed by a tope of $M$. See \Cref{fig:orientedvectormatroid} for an example.
\end{remark}

% REMA: reorientation and s-acyclic
\begin{remark}[Reorientation via sign vectors]\label{rema:reorientation}
In \Cref{defi:reorientation} we explained how to obtain the reoriented matroid $\reorientedmatroid{A}{M}$ of $M$ according to a subset $A\subset E$. However, the set $A$ defines a signed set $\tilde{A}$ by $\tilde{A}^-=A$ and $\tilde{A}^+ = E\setminus A^-$. Using the identification in \Cref{rema:signvectorsandsignedsets}, $\tilde{A}$ provides a sign vector $s(\tilde{A})$ defined by $s(\tilde{A})_e=\pm$ if and only $e\in\tilde{A}^\pm$. Multiplying a sign vector $v\in\svector$ with $s(\tilde{A})$ can be understood as ''changing signs of $v$ at coordinates indexed by $A$``. Let $C\in\circuits$ be a signed circuit and let $s_C=s(C)$ denote the sign vector associated to $C$. Then, the reoriented signed circuit $\reoriented{A}{C}$ equals the signed circuit $C\left(s(\tilde{A})\cdot s_C\right)$. By abuse of notation we write $s(\tilde{A})=\reoriented{A}{s}$ to indicate that we want to ``change signs at $A$''. As covectors are sign vectors, the reorientation of a covector $v\in\cvector{M}$ with respect to $A\subset E$ means switching signs at $A$, i.e. $v\in\cvector{M}$ translates to $\reoriented{A}{s}\cdot v\in\cvector{\reorientedmatroid{A}{M}}$ (cf. \cite[Lemma 4.18]{bjoernerlasvergnassturmfelswhiteziegler} and \Cref{rema:signoperations}).
\end{remark}

\subsection{Signed Bergman Fans}\label{subsec:signedbergmanfans}

% DEFI: s-acyclic
\begin{definition}[s-acyclic matroid]\label{defi:sacyclic}
Let $M$ be an oriented matroid and $s\in\svector$ a pure sign vector. We call the oriented matroid $M$ \textit{$s$-acyclic} if there is no circuit $C\in\circuits$ with $s_C\subseteq s$.
\end{definition}

\begin{remark}
Notice that $s_C\subseteq s$ is equivalent to $s_C^+ \subseteq s^+$ and $s_C^-\subseteq s^-$ (cf. \Cref{rema:signoperations}), i.e. for each $e\in E$ we have $\left(s_c\cdot s \right)_e\in\{0,+\}$. This allows to reformulate the definition in terms of sign vectors as follows: let $s\in\svector$ be a pure sign vector. Then, $M$ is $s$-acyclic if $s_C\cdot s\nsubseteq \left( + \right)^n$ for all $C\in\circuits$. 
\end{remark}

% REM: equivalent definition of s-acyclic 
\begin{remark}\label{rema:allpositive}
The situation for $s=(+)^n$ is well-known and was studied in \cite{ardilaklivanswilliams}, \cite{bjoernerlasvergnassturmfelswhiteziegler}. There, an oriented matroid $M$ is called \textit{acyclic} if there is no all-positive circuit in $M$. This case is covered by \Cref{defi:sacyclic}.
% In case of $s=(+)^n$ we require $M$ to be acyclic, i.e. there is no all-positive circuit in $M$. This situation is subject of study in \cite{ardilaklivanswilliams}, \cite{bjoernerlasvergnassturmfelswhiteziegler}.
\end{remark}

% LEMM: M s-acyclic <==> -A^M is acyclic <==> s is a tope of M
\begin{lemma}\label{lemm:acyclic}
Let $M$ be a matroid on $E=\groundset{n}$, $A\subseteq E$ a subset and $s={}_{-A}s\in\mathcal{S}$ the pure sign vector associated to $A$ for reorientation, i.e. $s^-=A$ and $s^+=E\setminus A$. Then:
\begin{equation}\notag
M~\text{is}~s\text{-acyclic}\quad\Longleftrightarrow\quad {}_{-A}M~\text{is acyclic}\quad\Longleftrightarrow\quad s~\text{is a tope of}~M.
\end{equation}
\end{lemma}
\begin{proof}
By definition, $M$ is $s$-acyclic if and only if $s_C\cdot s\nsubseteq (+)^n$ for all $C\in\circuits$. Due to the reorientation via sign vectors, we have $s\cdot s_C = {}_{-A}s\cdot s_C$, which is the sign vector of the reoriented circuit ${}_{-A}C$. All circuits of ${}_{-A}M$ are of the form ${}_{-A}C$ for all $C\in\circuits$. Thus, $M$ is $s$-acyclic is equivalent to ${}_{-A}C$ is not all-positive for all $C\in\circuits$. Equivalently, ${}_{-A}M$ is acyclic.\\
Note that by \cite[Proposition 3.4.8]{bjoernerlasvergnassturmfelswhiteziegler}, an oriented matroid $M$ is acyclic if and only if $(+)^n$ is a tope. We know that signed circuits as well as covectors of ${}_{-A}M$ are obtained from $M$ by multiplying with $s={}_{-A}s$. Hence, the above equivalence translates to ${}_{-A}M$ is acyclic if and only if $s\cdot (+)^n$ is a tope. To see the second equivalence, suppose that ${}_{-A}M$ is acyclic. Equivalently, $(+)^n$ is a tope of ${}_{-A}M$. Applying $s$ for reorientation means that $s\cdot (+)^n = s$ is a tope of ${}_{-A}\left( {}_{-A}M \right) = M$.
\end{proof}

\begin{remark}[Oriented initial matroids]
Oriented initial matroids are identically defined as unoriented initial matroids. In short, we consider an element $w\in\R^n$ as a weight function on $E$. The \textit{signed initial circuit} of a signed circuit $C\in\circuits$ is defined by the sets of positive/negative elements \[\initialcircuit{w}{C}^\pm=\left\{j\in C^\pm |w_j = \max_{i\in \underline{C}}\left\{w_i\right\}\right\}.\] In particular, for a signed circuit $C\in\circuits$ we have $\underline{\initialcircuit{w}{C}}=\initialcircuit{w}{\underline{C}}$. The collection of inclusion-minimal signed initial circuits $\initialcircuit{w}{C}$ for all $C\in\circuits$ is denoted by $\initialcircuits{w}$. The oriented initial matroid $M_w$ is defined by the signed initial circuits $\initialcircuits{w}$ (i.e. it is an oriented matroid according to \Cref{defi:orientedmatroid}, see \cite[Proposition 2.3]{ardilaklivanswilliams} for a proof).
\end{remark}

\begin{remark}[Flags of subsets and weight classes.]
For a fixed $w\in\R^n$ let $\flag(w)$ denote the flag of subsets $\emptyset \subset F_1 \subset \ldots \subset F_k=E$  such that $w$ is constant on $F_{i+1}\setminus F_i$ and $w_{F_i\setminus F_{i-1}} < w_{F_{i+1}\setminus F_i}$. The weight class of $w$ is the set of $v\in\R^n$ such that $\flag(w) = \flag(v)$.
\end{remark}

Since $M_w$ depends only on the flag $w$ is in we also refer to this initial matroid as $M_\flag$.

% DEF: s-flats and s-flags
\begin{definition}[$s$-flat/$s$-flag]\label{defi:sflatsflag}
Let $M$ be an oriented matroid on $E$ and $s\in\topes$ a tope. A flat $F$ of $M$ is called \textit{$s$-flat} if there is a covector $v\in\cvector{M}$ such that $v\subseteq s$ and $F=v^0$. A flag of flats $\flag$ is called \textit{$s$-flag} if all flats of $\flag$ are $s$-flats. We define $F_{i,j}=F_i\setminus F_{j}$ for all $0\leq j\leq i\leq k$ where $F_0=\emptyset$.
\end{definition}

% \begin{definition}\label{defi:flatdifference}
% Let $s\in\topes$ be a sign vector and $\flag=(F_1,\ldots,F_k)\triangleleft M$ be an $s$-flag of flats of a matroid $M$.
% \end{definition}

\begin{remark}\label{rema:chaincovectors}
Suppose $\flag=(F_1,\ldots,F_k)\triangleleft M$ is an $s$-flag of an oriented matroid for a tope $s\in\topes$. Let $v_1,\ldots,v_k\in\cvector{M}$ be the set of covectors such that $v_i^0=F_i$ and $v_i\subseteq s$. As $v_{i}\subseteq s$ and $F_{i}\subset F_{i+1}$ for all $i$, we have $v_{i+1}\subseteq v_i$. Thus, $s$-flags correspond to chains of covectors ordered by ``$\subseteq$'' (cf. \Cref{rema:signoperations}) ending with $s$. Moreover, $v_i$ coincides with $v_{i+1}$ at all coordinates where $v_{i+1}$ is non-zero. Note that $v_i\neq 0$ at $F_{k+1,i}$ and $v_i$ differs from $v_{i+1}$ at $F_{i+1,i}$.
\end{remark}

\begin{remark}[Big face lattice]\label{rema:bigfacelattice}
The collection of flats of $\umatroid{M}$ equals the collection of zero sets $v^0$ of covectors $v\in\cvector{M}$ (\cite[Proposition 4.1.13]{bjoernerlasvergnassturmfelswhiteziegler}). The covectors $\cvector{M}$ of $M$, equipped with the induced partial order $\subseteq$ of $\svector$ (cf. \Cref{rema:signoperations}) and bottom/top elements $\hat{0}$/$\hat{1}$, form a lattice $\biglattice\left( M \right)=\left( \cvector{M}\cup\left\{\hat{0},\hat{1}\right\},\subseteq \right)$ called the \textit{big face lattice} of $M$. %Let $\slattice{s}$ denote the sublattice of $\biglattice$ of all $s$-flats. We call $\slattice{s}$ the \textit{s-lattice} (of $\biglattice$).
\end{remark}

\begin{example}
Recall the oriented vector matroid $M[A]$ given by the matrix
\begin{equation}
A=\begin{pmatrix}
  1 & 0 & 1 & -1 & -2\\
  0 & 1 & 1 & -1 & -1
  \end{pmatrix}.
\end{equation}
The point configuration $\support$ defined by the columns of $A$ is shown in \Cref{fig:orientedvectormatroid}. The subdivision of $\Dual{\R^2}$ into topes of $M[A]$ is shown in \Cref{fig:orientedvectormatroid}. Consider the chains of covectors \[\mathcal{S}:~v_1=\begin{bmatrix}
0\\
0\\
0\\
0\\
0
\end{bmatrix}\subseteq
v_2=\begin{bmatrix}
-\\
+\\
0\\
0\\
+
\end{bmatrix}\subseteq
v_3=\begin{bmatrix}
-\\
+\\
+\\
-\\
+
\end{bmatrix}
\quad\And{}\mathcal{S}':~v_1=\begin{bmatrix}
0\\
0\\
0\\
0\\
0
\end{bmatrix}\subseteq
v_2=\begin{bmatrix}
-\\
+\\
0\\
0\\
+
\end{bmatrix}\subseteq
v_3'=\begin{bmatrix}
-\\
+\\
-\\
+\\
+
\end{bmatrix}.\label{eq:chainsofcovectors}\] Let $\flag=(v_1^0,v_2^0,v_3^0)$ and $\flag'=(v_1^0,v_2^0,{v_3^0}')$ denote the flags of flats arising from the zero sets of the covectors of $\mathcal{S}$ and $\mathcal{S}'$, cf. \Cref{eq:chainsofcovectors}. Note that
\begin{equation}\notag
\flag=\flag'=(F_1,F_2,F_3)\With{} F_1=E,~F_2=\{3,4\}\And F_3=\emptyset,
\end{equation}
whereas $\mathcal{S}\neq\mathcal{S}'$. If we pick $s=v_3\in\pvector{5}$ (or $s'=v_3'\in\pvector{5}$) then $\flag$ is an $s$-flag ($s'$-flag respectively). For a fixed tope $s\in\topes$, the $s$-flats correspond to faces of the cell corresponding to $s$ in $\Dual{\R^2}$. Hence, all $s$-flags of $M$ correspond to collections of faces of the cell of $\Dual{\R^2}$ dual to $s$ ordered by inclusion.
\end{example}

\begin{definition}[Signed Bergman fan]\label{defi:signedBergmanfan}
The \textit{signed Bergman fan} of an oriented matroid $M$ on the ground set $E=\groundset{n}$ with respect to a pure sign vector $s\in\svector$ is defined by \[\signedbergmanfan{s}{M} := \left\{w\in\R^n | M_w~\text{is}~s\Acyclic\right\}.\]
\end{definition}

The signed Bergman fan $\signedbergmanfan{s}{M}$ with respect to $s=(+)^n$ is called \textit{positive Bergman fan} $\mathcal{B}^+\left( M \right)$ (\cite{ardilaklivanswilliams}). There, a covector $v\in\cvector{M}$ is called \textit{positive} if $v^-=\emptyset$ and a flat is positive if it is a $\left( + \right)^n$-flat. For the positive Bergman fan, i.e. $s=(+)^n$, we have the following theorem:
% THEO: positive Bergman fan
\begin{theorem}[{\cite[Theorem 3.4]{ardilaklivanswilliams}}]\label{theo:positivebergmanfan}
Given an oriented matroid $M$ and $w\in\R^n$ which corresponds to a flag $\flag=\flag\left( w \right)$, the following are equivalent:
\begin{itemize}
\item[1.] $\mathcal{M}_\flag$ is acyclic.
\item[2.] For each signed circuit $C$ of $M$, $\initialform{w}{C}$ contains a positive and negative element of $C$.
\item[3.] $\flag$ is a flag of positive flats of $M$.
\end{itemize}
\end{theorem}
Now, we generalize \Cref{theo:positivebergmanfan} for any pure sign vector $s\in\svector$:
% THEO: signed Bergman fan
\begin{theorem}\label{theo:signedbergmanfan}
Let $M$ be an oriented matroid on $E$, $s\in\svector$ a pure sign vector, $w\in\R^n$ and $\flag=\flag(w)$ the corresponding flag. The following are equivalent:
\begin{enumerate}
\item[1.] $M_\flag$ is $s$-acyclic.
\item[2.] For all $C\in\circuits$, $w$ attains its maximum at $\left( s\cdot s_C \right)^+$ and $\left( s\cdot s_C \right)^-$.
\item[3.] $\flag$ is a $s$-flag of flats of $M$.
\end{enumerate}
\end{theorem}
\begin{proof}
At first, we set $A=s^-$, i.e. we reorientate with respect to $A$. Note that reorienting and initializing a matroid $M$ commutes, i.e. for $A\subseteq E$ and $\flag=\flag(w)$ for $w\in\R^n$ we have
\begin{equation}\notag
{}_{-A}\left( M_\flag \right) = \left( {}_{-A}M \right)_\flag.
\end{equation}
The reason for this is that $w$ picks the elements of a (signed) circuit $C$ where $w$ is maximal independently from the signs. Circuits of the underlying matroid $\umatroid{M}$ remain invariant under reorientation. Also note that covectors $v\in\cvector{M}$ translate to covectors $s\cdot v\in\svector_{{}_{-A}M}$ of ${}_{-A}M$ (\cite[Section 3 + Lemma 4.18]{bjoernerlasvergnassturmfelswhiteziegler}) and finally, ${}_{-A}\left( {}_{-A}M \right)=M$ and $s\cdot \left( s\cdot v \right)=v$.\\
For $1.\Rightarrow 2.$ suppose $M_\flag$ is $s$-acyclic. Thus, $\reorientedmatroid{A}{\left(M_\flag\right)}=\left(\reorientedmatroid{A}{M}\right)_\flag$ is acyclic (cf. \Cref{lemm:acyclic}) and we have $s(\initialcircuit{w}{\reoriented{A}{C}})\nsubseteq(+)^n$ for all circuits $C\in\circuits$. Since reorientation commutes with initialization, we have $\reoriented{A}{s}\cdot s(\initialcircuit{w}{C})\nsubseteq(+)^n$ for all $C\in\circuits$. Consequently, for all $C\in\circuits$ exist $e,f\in E$ such that (w.l.o.g.) $(s)_e (s(\initialcircuit{w}{C}))_e = -$ and $(s)_f (s(\initialcircuit{w}{C}))_f = +$. Since $\initialcircuit{w}{C}\subseteq C$ we conclude that for all $C\in\circuits$ exist $e,f\in E$ such that $(s)_e (s(C))_e=-$ and $(s)_f (s(C))_f =+$. In other words, for all $C\in\circuits$ holds that $w$ attains its maximum at $(s\cdot s(C))^+$ and $(s\cdot s(C))^-$.\\
Vice versa, suppose that $w$ attains its maximum at $(s\cdot s(C))^+$ and $(s\cdot s(C))^-$ for all $C\in\circuits$. Since $s\cdot s_C =\reoriented{A}{s}\cdot s_C$ we conclude that $(\reoriented{A}{s}\cdot s(\initialcircuit{w}{C}))^\pm\neq\emptyset$ for all $C\in\circuits$. Hence, $s\cdot s(\initialcircuit{w}{C}) = s(\initialcircuit{w}{\reoriented{A}{C}})\nsubseteq(+)^n$ for all $C\in\circuits$, i.e. $\left(\reorientedmatroid{A}{M}\right)_\flag$ is acyclic. \Cref{lemm:acyclic} implies that $M_\flag$ is $s$-acyclic.\\
For $1.\Rightarrow 3.$ suppose that $M_\flag$ is $s$-acyclic. Hence, $\left(\reorientedmatroid{A}{M}\right)_\flag$ is acyclic (\Cref{lemm:acyclic}). Thus, $\flag$ is a positive flag of flats of $\reorientedmatroid{A}{M}$ (\Cref{theo:positivebergmanfan}). Let $\flag=(F_1,\ldots,F_l)$ be the flag of flats and $\{v_1,\ldots,v_l\}\subset\cvector{\reorientedmatroid{A}{M}}$ the covectors such that $v_i^0=F_i$ for $1\leq i \leq l$. According to \Cref{rema:reorientation}, $\{s\cdot v_1,\ldots,s\cdot v_l\}\subset\cvector{M}$ is a set of covectors satisfying $(s\cdot v_i)^0=F_i$ for $1\leq i\leq l$. To see this note that $s$ is a pure sign vector, i.e. $s^0=\emptyset$. Thus, $(s\cdot v_i)_e=0$ if and only if $(v_i)_e=0$. Moreover, $v_i\subset(+)^n$, i.e. $v_i^-=\emptyset$ for $1\leq i\leq l$. Thus, $s_i=(s\cdot v_i)\subseteq s$ and, therefore, $\{s_1,\ldots,s_l\}=\{s\cdot v_1,\ldots,s\cdot v_l\}\subset\cvector{M}$ is a set of covectors such that $\flag$ is a $s$-flag.\\
Vice versa, let $\flag$ be a $s$-flag, i.e. $\flag=(F_1,\ldots,F_l)$ and there is a set $\{v_1,\ldots,v_l\}\subset\cvector{M}$ such that $F_i=v_i^0$ and $v_i\subseteq s$. Using \Cref{rema:reorientation} again, we conclude that $\{s\cdot v_1,\ldots,s\cdot v_l\}=\{s_1,\ldots,s_l\}\subset\cvector{\reorientedmatroid{A}{M}}$ is a set of covectors such that $s_i^0=(s\cdot v_i)^0 =v_i^0$ and $s_i\subset (+)^n$ for $1\leq i\leq l$. Hence, $\flag$ is a positive flag of $\reorientedmatroid{A}{M}$. Due to \Cref{theo:positivebergmanfan} this implies $\left(\reorientedmatroid{A}{M}\right)_\flag$ is acyclic. By \Cref{lemm:acyclic} this is equivalent to $M_\flag$ is $s$-acyclic.
%  old
% At first we show $1.\Leftrightarrow 2.$.
% $M_\flag$ is $s$-acyclic if and only if ${}_{-A}\left( M_\flag \right)=\left( {}_{-A}M \right)_\flag$ is acyclic. Hence, $\initialform{w}{{}_{-A}C}=\initialform{w}{C\left(s\cdot s_C\right)} = C\left(s\cdot s_{\initialform{w}{C}}\right) \nsubseteq (+)^n$ for all $C\in\circuits$. Thus, we have (w.l.o.g. considering $s^+$): $|s^+\cap \left( s_{\initialform{w}{C}} \right)^+|\geq 1$ and $|s^+\cap \left( s_{\initialform{w}{C}} \right)^-|\geq 1$. Since $\initialform{w}{C}\subseteq C$ it holds $|s^+\cap s_C^+|\geq 1$ and $|s^+\cap s_C^-|\geq 1$.\\
% Now, we show $1.\Leftrightarrow 3.$. $M_\flag$ is $s$-acyclic $\Leftrightarrow$ ${}_{-A}\left( M_\flag \right)$ is acyclic. Equivalently $\left( {}_{-A}M \right)_\flag$ is acyclic. This is equivalent to $\flag$ is a positive flag of flats of ${}_{-A}M$. Equivalently $\flag$ is a $s$-flag of $M$. The last equivalence holds since covectors of ${}_{-A}M$ translate to covectors of $M$ by multiplying with $s$. Every positive covector $v$ translates to $s\cdot v \subseteq s$ satisfying $\left( s\cdot v \right)^0 = v^0$. Hence, a positive flat becomes a $s$-flat.
\end{proof}
% LEMM: s not a tope then B^s(M) empty
\begin{corollary}
Let $M$ be an oriented matroid and $\topes$ its set of topes (cf. \Cref{rema:topes}). Then:
\begin{equation}\notag
s\notin\topes\quad\Rightarrow\quad\signedbergmanfan{s}{M}=\emptyset.
\end{equation}
\end{corollary}
\begin{proof}
Assume $\signedbergmanfan{s}{M}\neq\emptyset$. Then, there is an element $w\in\signedbergmanfan{s}{M}$ such that $M_w$ is $s$-acyclic. By \Cref{theo:signedbergmanfan} we know that $\flag\left( w \right)$ is a $s$-flag of $M$. Hence, $s\in\topes$. 
\end{proof}

From \Cref{theo:signedbergmanfan} we can immediately deduce the following
\begin{corollary}
Let $M$ be an oriented matroid. Then, the signed Bergman fan of $M$ with respect to $s\in\pvector{}$ is the union of weight class defined by $s$-flags of flats of $M$:
\begin{equation}\notag
\signedbergmanfan{s}{M}=\bigcup_{\substack{\flag\triangleleft\umatroid{M}:\\\flag~\text{is an}~s\text{-flag}}}\sigma_\flag.
\end{equation}
\end{corollary}

Let us turn back to real tropicalizations of real linear spaces over $\RealField$. The positive part of $\realtrop\left(\Var{\Ideal{I}}\right)$ is well-understood (\cite[Proposition 4.1]{ardilaklivanswilliams}). For $\Ideal{I}\subset\RealLauring{n}$ a linear ideal with constant coefficients and $M$ the associated oriented matroid we have $\realtrop\left(\Var{\Ideal{I}}\right)\cap\left( (+)^n\times\R^n \right) = \signedbergmanfan{+}{M}$. We generalize this statement for arbitrary pure sign vectors $s\in\svector$:
% % THEO: signed Bergman fan
\begin{theorem}\label{theo:realtropicallinearspace=signedbergmanfan}
Let $\Ideal{I}\subset\RealLauring{n}$ be a linear ideal with constant coefficients and $M$ the associated oriented matroid with signed circuits $\circuits$. Let $s\in\svector$ be a pure sign vector. Then:
\begin{equation}\notag
\realtrop\left(\Var{\Ideal{I}}\right)\cap\left( s\times\R^n \right) = \signedbergmanfan{s}{M}.
\end{equation}
\end{theorem}
\begin{proof}
The circuits $\circuits$ form a tropical basis of $\realtrop\left(\Var{\Ideal{I}}\right)$ (\Cref{theo:realtropicalbasis}). Therefore, we show
\begin{equation}\notag
\left(\bigcap_{C\in\circuits} \RealTh{l_C}\right)\cap\left( s\times\R^n \right) = \signedbergmanfan{s}{M}
\end{equation}
for arbitrary pure sign vectors $s\in\svector$. Recall that $l_C$ is the real tropical linear form obtained uniquely from $C$ (cf. \Cref{rema:orientedmatroidassociatedtoideal}). We write it in its simplest form, i.e. $l_C=\sum_{j\in\underline{C}} p_j w_j\in\RealTropaffring{n}$ with $p_j\in\{0^\pm\}$ for all $j\in\underline{C}$. Furthermore, recall that we denoted the sign vector obtained from $C$ by $s_C$. Here we get $(s_C)_j = s(p_j)\in\{\pm\}$ for all $j\in\underline{C}$ and $(s_C)=0$ for all $j\notin\underline{C}$. Now, suppose $w\in \left(\bigcap_{C\in\circuits} \RealTh{l_C}\right)\cap\left( s\times\R^n \right)$. The sign vector of $w$ is $s$, in detail $s(w_i)=s_i$ for all $i$. As a consequence we identify the modulus of $w$ with $w$. Since $w\in\bigcap_{C\in\circuits} \RealTh{l_C}$ it follows that, by definition, for all circuits $C\in\circuits$ exist $i,j\in\underline{C}$ such that $s(p_i)s(w_i) \neq s(p_j)s(w_j)$ and $w_i = w_j \geq w_k\forall k$. Equivalently for all circuits $C\in\circuits$ exist $i,j\in\underline{C}$ such that (w.l.o.g.) $i\in\left( s\cdot s_C \right)^-$, $j\in\left( s\cdot s_C \right)^+$ and $w_i=w_j\geq w_k\forall k$. This is precisely statement 2 of \Cref{theo:signedbergmanfan}, i.e. $M_w$ is $s$-acyclic.
\end{proof}

%% file: sec_3.tex
% !TEX root = /Users/christianjurgens/Dropbox/UniSL/Projekte/Real singular plane tropical curves/Paper/ms.tex

%%%%%%%%%%%%% SINGULAR REAL TROPICAL CURVES %%%%%%%%%%%%%
\section{Singular Real Plane Tropical Curves}\label{sec:realsingularplanetropicalcurves}

This section deals with the classification of singular real plane tropical curves.

\begin{notation}\label{nota:sec3}
Let $\Delta\subset\Z^n$ be a convex lattice polytope and $\support=\Delta\cap\Z^n=\{\alpha_1,\ldots,\alpha_m\}$ its set of lattice points. By $\psi_\support:\RealTorus{n}\rightarrow\RealTorus{m}$ we denote the monomial map according to $\support$. For a set $B$ we denote by $p_B:\R^n\rightarrow\R_B$ the coordinate projection onto the coordinates indexed by $B$. Consider a generic real Laurent polynomial $F=\sum_i y_i x^{\alpha_i}\in\RealGenericringnvars{}$ that is linear in the coefficients. We write $R=\Realcoefficientring{m}$ for the polynomial ring forming the coefficients. We denote the Laurent polynomial obtained from $F$ with fixed coefficients $a\in\RealTorus{m}$ by $F_a=\sum_i a_i x^{\alpha_i}\in\RealLauring{n}$ and $F(p)=\sum_i y_i p^{\alpha_i}\in\Realcoefficientring{m}$ denotes the polynomial obtained from $F$ by evaluating at $p\in\RealTorus{n}$. In the following we write
\begin{equation}
\Ideal{I}=\langle F\left( \ones{n} \right),~\frac{\partial F}{\partial x_{1}}(\ones{n}),\ldots,\frac{\partial F}{\partial x_{n}}(\ones{1})\rangle\subset R\label{eq:realidealsingsatone}
\end{equation}
for the ideal generated by $F$ and its partial derivatives $\frac{\partial F}{\partial x_{i}}$ evaluated at $\ones{n}$. Let $A\in\Z^{n\times m}$ be the matrix representation of the point configuration $\support$. Let $A'\in\Z^{n+1\times m}$ be the matrix containing the coefficients of the generators of $\Ideal{I}$: 
\begin{equation}
A'=	\begin{bmatrix} 1 		& \cdots	& 1 \\ 
			\alpha_1 	& \cdots	& \alpha_m
	\end{bmatrix}=\begin{bmatrix}
		      \ones{m}^\top\\
		      A
		      \end{bmatrix}\in\Z^{3\times m}.\label{eq:realshiftedsupportmatrix}
\end{equation}
The columns of $A'$ correspond to the shift of the points of $\support$ into $\R^3$.
\end{notation}

We study the family of real Laurent polynomials that provide a singular real plane curves with a singularity fixed in $\ones{2}$:
\begin{equation}
\Realsingsat{\ones{2}}=\left\{a\in\Proj\left(\RealTorus{m}\right):\Var{F_a}~\text{is singular at}~\ones{2}\right\}=\Var{\Ideal{I}}=\ker(A').\label{eq:realsings}
\end{equation}

\subsection{Tropicalizations of Real Hypersurfaces with a Singularity in a Real Point}\label{subsec:realtropicaldiscriminant}

By definition, for any $a\in\Realsingsat{\ones{2}}$ we know that $\Var{F_a}$ is singular at $\ones{2}$. Let $\ones{2}\neq p\in\RealTorus{n}$ be any other torus point. It is not hard to see that
\begin{equation}\notag
F_a=\sum_{\alpha\in\support}a_\alpha x^\alpha~\text{is singular at}~p\quad\Leftrightarrow\quad F_{a\cdot\psi_\support(p)}=\sum_{\alpha\in\support}a_\alpha p^\alpha x^\alpha~\text{is singular at}~\ones{2}.
\end{equation}
In the complex case, it is well-known that $\complextrop(\psi_\support)(w)=A^\top w$. For an element $p\in\ComplexTorus{n}$ and $\complextrop(p)=-\val(p)=q$ we have $\complextrop(\psi_\support(p))=(\complextrop(p^\alpha))_{\alpha\in\support}=(\sk{q}{\alpha})_{\alpha\in\support}=A^\top q$. In the real case, we have to take signs into account, i.e. for $p\in\RealTorus{2}$ and $q=\realtrop(p)=(s(p),-\val(p))$ we have $\realtrop(\psi_\support(p))=(s(p^{\alpha}),-\val(p^\alpha))_{\alpha\in\support}=(s(p_x)^{\alpha_x}s(p_y)^{\alpha_y},\sk{|q|}{\alpha})_{\alpha\in\support}$. Thus, the real tropicalization of $a\cdot\psi_\support(p)$ where $\realtrop(a)=(s(a_\alpha),-\val(a_\alpha))_{\alpha\in\support}=(s_\alpha,b_\alpha)_{\alpha\in\support}\in\RealTropicalgroup^m$ is \[\realtrop(a\cdot\psi_\support(p)) =\left( s_\alpha s(p_x)^{\alpha_x}s(p_y)^{\alpha_y},b_\alpha+\sk{(-\val(p))}{\alpha} \right)_{\alpha\in\support}\in\RealTropicalgroup^m.\] We see that $\realtrop(a\cdot\psi_\support(p))=\realtrop(a)\odot_{\R}\realtrop(\psi_\support(p))$, i.e. the modulus of $\realtrop(a)$ is shifted by an element in the row space of $A$ (as in the complex case) and we perform a sign vector multiplication on the signs of $\realtrop(a)$ that are defined by $\psi_\support$.
\begin{remark}\label{rema:realtropicaldiscriminant}
The paragraph before explains that the product \[\realtrop(\Realsingsat{\ones{n}})\odot_{\R}\realtrop(\Image(\psi_\support))\] parametrizes real tropicalizations of real plane hypersurfaces with a singularity in a real torus point. If we restrict the monomial map $\psi_{\support}$ to $\{\pm 1\}^{n}\cong\{\pm\}^{n}$ and define $G=\{\psi_\support(v):v\in\{\pm\}^n\}\subset\pvector{m}$ then we have $\realtrop(\Image(\psi_\support))=G\times\rowspace(A)$.
\end{remark}

\begin{definition}[Lineality group]\label{defi:linealitygroup}
We call $G\times\rowspace(A)$ \textit{lineality group}.
\end{definition}

\subsection{Real Plane Tropical Curves and the Signed Secondary Fan}\label{subsec:realplanetropicalcurvesandsignedsecondaryfan}

Singular plane tropical curves over $\ComplexField$ were studied and characterized in \cite{markwigmarkwigshustin}. We adapt their methods for the real case and point out the differences owed to the signs. In this section, we introduce the signed secondary fan and examine its relationship to real plane tropical curves. To begin with, recall basics of real plane tropical curves. A real tropical Laurent polynomial $f=\bigoplus_{i} p_i w^{\alpha_i}\in\RealTroplauringtwovars$ with support $\support$ provides a piecewise affine linear function $|f|(w)=\max_i\{|p_i|+\sk{w}{\alpha_i}\}$ (cf. \Cref{defi:realtropicalpolynomial}) called \textit{modulus of $f$}. Basically, $|f|$ forgets about the signs. The tropical curve defined by $|f|$ is dual to the regular marked subdivision of $\Delta=\Newt{f}$ (explained below) with respect to the coefficient moduli of $f$. With regard to $f$, this subdivisions comes with signs.

\begin{definition}[Signed marked subdivision]\label{defi:signedmarkedsubdivision}
A \textit{signed marked polytope} $(P,Q,s_Q)$ consists of a marked polytope $(P,Q)$ and a sign vector $s_Q\in\{\pm\}^{|Q|}$ such that $\alpha\in Q$ has a sign $s_\alpha$. A \textit{signed marked subdivision} is a set of signed marked polytopes, $T=\{(P_i,Q_i,s_{Q_i}):i=1,\ldots,k\}$, satisfying
\begin{itemize}
\item the collection of marked polytopes $(P_i,Q_i)$ with $i\in\groundset{k}$ forms a marked subdivision, and
\item signs of marked polytopes are compatible, i.e. $p_{Q_i\cap Q_j}(s_{Q_i})=p_{Q_i\cap Q_j}(s_{Q_j})$ for all $i,j\in\groundset{k}$.
\end{itemize}
If we forget about signs we obtain a marked subdivision denoted by $|T|$. As for marked subdivisions, we call the collection of $P_i$ without markings and signs the \textit{type} of $T$. The boundary of $T$ is $\partial |T|$.
\end{definition}

Let $w\in\RA{}$ be a vector. It defines heights on the points $\support$. The projection of the upper facets of the convex hull \[\conv{(\alpha_{i},w_{i}):i\in\groundset{m}}\] defines a subdivision of $\Delta$ called \textit{regular marked subdivision}. A point $\alpha_{i}$ is \textit{marked} if $(\alpha_{i},w_{i})$ is contained in an upper facet. By $\secondaryfan{\support}$ we denote the complete fan supported on $\RA$ whose cones provide equivalence classes of regular marked subdivisions of $\Delta$. If $w\in\RA{}$ provides the subdivision $T$, then we denote the cone $w$ is in by $\sigma(w)=\sigma_{T}$. An element $(s,u)\in\signedsecondaryfan{\support}$ provides a signed regular marked subdivision $T$: we get a regular marked subdivision $|T|=\{(P_i,Q_i):i=1,\ldots,k\}$ of $\Delta$ by the modulus $u\in\secondaryfan{\support}$ and equip the vertices with the signs according to $s$, i.e. $\alpha\in Q_i$ gets the sign $s_\alpha$.

\begin{remark}[Klein group]\label{rema:kleingroup}
Note that we have$\{\pm\}^2=\{(+,+),(+,-),(-,+),(-,-)\}$. This 4-element set forms the \textit{Klein group} $\Kleingroup$. In general, the set of pure sign vectors $\pvector{m}$ forms a multiplicative group. If we identify $\{\pm\}\cong\{1,-1\}$ we can think of $\psi_\support:\pvector{2}\rightarrow\pvector{m}$ as a group homomorphism where $\psi_\support$ denotes the monomial map according to $\support$ (cf. \Cref{nota:sec3}). Let \[G=\{\psi_\support(v):v\in\{\pm\}^2\}\subset\pvector{m}\] denote the image of $\Kleingroup$. We have $G\cong\Kleingroup$ if and only if the cardinality of $G$ is 4. In the following we refer to $\pvector{2}$ as the Klein group. We can associate sublattices of $\Z^2$ to elements of the Klein group: to $v\in\Kleingroup$ we associate the sublattice $\{p\in\Z^2:v^p=+\}$ (cf. \Cref{fig:Kleingrouplattices}).
\end{remark}

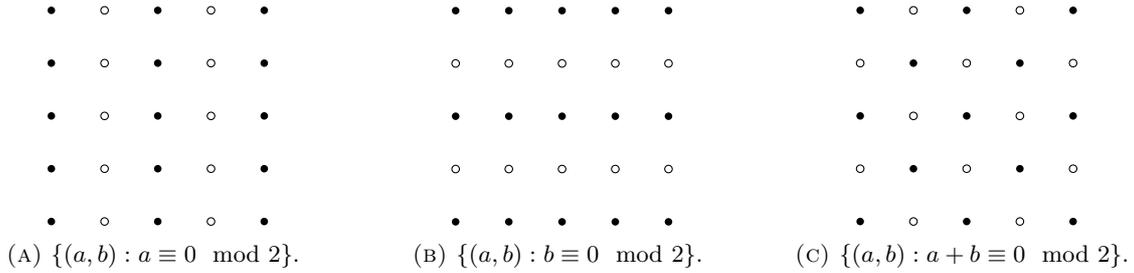
\begin{figure}
\subcaptionbox{$\{(a,b):a\equiv 0\mod 2\}$.}[0.3\textwidth]{
\begin{tikzpicture}
\def\shift{0.7}
\foreach \x in {1,...,5}
{
  \foreach \y in {1,...,5}
  {
    \pgfmathparse{int(Mod(\x,2))}
    \let\result\pgfmathresult
    \ifnum\result=0 
      \coordinate[shape=circle,inner sep=1 pt,draw,label={[]}] (P\x\y) at (\x*\shift,\y*\shift);
    \else
      \coordinate[shape=circle,inner sep=1 pt,fill,label={[]}] (P\x\y) at (\x*\shift,\y*\shift);
    \fi
  }
}
\end{tikzpicture}
}\hfill
\subcaptionbox{$\{(a,b):b\equiv 0\mod 2\}$.}[0.3\textwidth]{
\begin{tikzpicture}
\def\shift{0.7}
\foreach \x in {1,...,5}
{
  \foreach \y in {1,...,5}
  {
    \pgfmathparse{int(Mod(\y,2))}
    \let\result\pgfmathresult
    \ifnum\result=0 
      \coordinate[shape=circle,inner sep=1 pt,draw,label={[]}] (P\x\y) at (\x*\shift,\y*\shift);
    \else
      \coordinate[shape=circle,inner sep=1 pt,fill,label={[]}] (P\x\y) at (\x*\shift,\y*\shift);
    \fi
  }
}
\end{tikzpicture}
}\hfill
\subcaptionbox{$\{(a,b):a+b\equiv 0\mod 2\}$.}[0.3\textwidth]{
\begin{tikzpicture}
\def\shift{0.7}
\foreach \x in {1,...,5}
{
  \foreach \y in {1,...,5}
  {
    \pgfmathparse{int(Mod(\x+\y,2))}
    \let\result\pgfmathresult
    \ifnum\result=0 
      \coordinate[shape=circle,inner sep=1 pt,fill,label={[]}] (P\x\y) at (\x*\shift,\y*\shift);
    \else
      \coordinate[shape=circle,inner sep=1 pt,draw,label={[]}] (P\x\y) at (\x*\shift,\y*\shift);
    \fi
  }
}
\end{tikzpicture}
}
\caption{Sublattices of $\Z^2$ corresponding to elements of the Klein group $\Kleingroup$.\label{fig:Kleingrouplattices}}
\end{figure}

\begin{remark}[Signed regular marked subdivisions and the Klein group]\label{rema:equivalentsignedregularmarkedsubdivisions}
Consider the signed regular marked subdivision $T$ obtained from $(s,u)\in\signedsecondaryfan{\support}$. Let $v\in\pvector{2}$ be arbitrary. Then, we denote the signed regular marked subdivision defined by the element $(s\cdot\psi_\support(v),u)\in\signedsecondaryfan{\support}$ by $T_v$. Note that $T_v$ consists of the identical marked polytopes as $T$, i.e. $|T_v|=|T|$. The signs of $T_v$ equal the signs of $T$ multiplied with $\psi_\support(v)$. In particular, we have $T=T_{(+,+)}$. Note that we get $T_v$ from $T$ by changing signs at all points contained in the sublattice corresponding to $v$ (cf. \Cref{rema:kleingroup} and \Cref{fig:realtropicalhypersurface}). We define an equivalence relation on $\pvector{m}$: two elements $s,s'\in\pvector{m}$ are called \textit{sign equivalent} or (\textit{$G$-equivalent}) if and only if there exists an element $v\in\Kleingroup$ such that $s=\psi_\support(v)\cdot s'$. Equivalently, we call $s,s'\in\pvector{m}$ sign equivalent if and only if $s=s'\in\pvector{m}\big/G$.

% if and only if $u=u'$ and there exists an element $v\in\Kleingroup$ such that $s=\psi_\support(v)\cdot s'$.
\end{remark}

\begin{definition}[Charts of real plane tropical curves]\label{defi:charts}
Let $f\in\RealTroplauring{n}$ be a real tropical Laurent polynomial and $C_f=\RealTh{f}\subset\RealTropicalgroup^n$ its associated real tropical hypersurface. A point $w\in C_f$ has $n$ signs, $s(w)=v\in\pvector{n}$, according to its $n$ coordinates. By $C_{f,v}$ (or $(C_{f})_{v}$) we denote the part of $C_f$ containing the solutions $w\in C_f$ with $s(w)=v$ and call $C_{f,v}$ \textit{$v$-chart} of $C_f$. In particular, we have $C_{f,v}\subset v\times\R^{n}$ (cf. \Cref{defi:chartsambientspace}).
\end{definition}

% \begin{remark}[Collection of signed regular marked subdivisions]
% Now, consider the element $(s',u)=(s\cdot\psi_\support(v),u)\in\signedsecondaryfan{\support}$ for some arbitrary $v\in\{\pm\}^2$. This element defines a curve $C_{f'}$ as well as a signed regular marked subdivision $T'$. We have $T_{(+,+)}'=T_v$: the signs of $T_v$ are precisely $s\cdot\psi_{\support}(v)$. The four occurring signed regular marked subdivisions obtained from $T$ and $T'$ are, therefore, identical, only the indices are permuted. The signed regular marked subdivision $T_{v'}'$ of $T'$ equals the signed marked subdivision arising from $T$ indexed by $v'\cdot v$.
% \end{remark}

\begin{remark}[Duality]\label{rema:duality}
An element $(s,u)\in\signedsecondaryfan{\support}$ provides a real tropical Laurent polynomial $f=\bigoplus_{\alpha\in\support} (s_\alpha,u_\alpha)w^\alpha\in\RealTroplauringtwovars$ and, therefore, a real plane tropical curve $C_f$. Moreover, $(s,u)$ induces a signed regular marked subdivision $T$ of $\Delta$. In the following we explain the duality of the curve $C_f$ and the signed regular marked subdivisions $T_{v}$ arising from $(s,u)\in\signedsecondaryfan{\support}$: first, note that $|T|$ is the regular marked subdivision of $\Delta$ according to the modulus $u$, i.e. $\Th{|f|}$ is dual to $|T|$ in the usual sense. In particular, vertices of $\Th{|f|}$ correspond to marked polygons and each edge $e$ of $\Th{|f|}$ corresponds to an edge $E$ of a polygon. Dual edges are perpendicular. An edge $e$ of $\Th{|f|}$ is unbounded if and only if the dual edge $E$ is contained in the boundary of $\Delta$.\\
Now, consider $v\times\R^2$ indexed by $v\in\{\pm\}^2$. Then, $T_v$ is the signed regular marked subdivision obtained from $T$ by adjusting signs with $\psi_\support$ (cf. \Cref{rema:equivalentsignedregularmarkedsubdivisions}): if $(P_i,Q_i,s_{Q_i})$ is a signed marked polygon of $T$ then $(P_i,Q_i,s_{Q_i}\cdot\psi_{Q_i}(v))$ is a signed marked polygon of $T_v$ where $(s_{Q_i}\psi_{Q_i}(v))_\alpha = (s_{Q_i})_\alpha v^\alpha$ for each $\alpha\in Q_i$. Hence, $C_{f,v}$ is dual to $T_v$ with respect to signs. In detail, $C_{f,v}$ contains a vertex dual to a  signed marked polygon $(P_i,Q_i,s_{Q_i}\cdot\psi_{Q_i}(v))$ if and only if there are two elements $\alpha,\alpha'\in Q_i$ such that $(s_{Q_i}\cdot\psi_{Q_i}(v))_\alpha\neq (s_{Q_i}\cdot\psi_{Q_i}(v))_{\alpha'}$. Note that this is precisely the sign condition of \Cref{defi:realtropicalhypersurface}. There is a (bounded or unbounded) edge $e$ in $C_{f,v}$ if and only if there is a dual edge $E$ contained in a signed marked polytope $(P_i,Q_i,s_{Q_i}\cdot\psi_{Q_i}(v))$ such that there are two elements $\alpha,\alpha'\in E\cap Q_i$ such that $(s_{Q_i}\cdot\psi_{Q_i}(v))_\alpha\neq (s_{Q_i}\cdot\psi_{Q_i}(v))_{\alpha'}$. Thus, $C_{f,v}$ contains those polyhedra of $\Th{|f|}$ that are dual to signed marked polytopes of $T_v$ with at least two vertices having different signs.
\end{remark}

\begin{example}
Consider $\support=\{x\in\Z^{2}:\sk{x}{\ones{2}}\leq 2,x\geq 0\}=\{\alpha_1,\ldots,\alpha_6\}$ and the element \[(s,u)=((+,-,+,-,+,+),(-1,0,0,-1,0,0))\in\signedsecondaryfan{\support}.\] The point $(s,u)$ defines the real tropical Laurent polynomial $f=-1^+\oplus0^-x\oplus0^+y\oplus-1^-x^2\oplus0^+xy\oplus0^+y^2$. The signed regular marked subdivisions $T_{v}$ of $\Delta$ induced by $(s,u)$ and $v\in\pvector{2}$ are shown in \Cref{fig:signedregularmarkedsubdivisions}. Note that $T=T_{(+,+)}$.
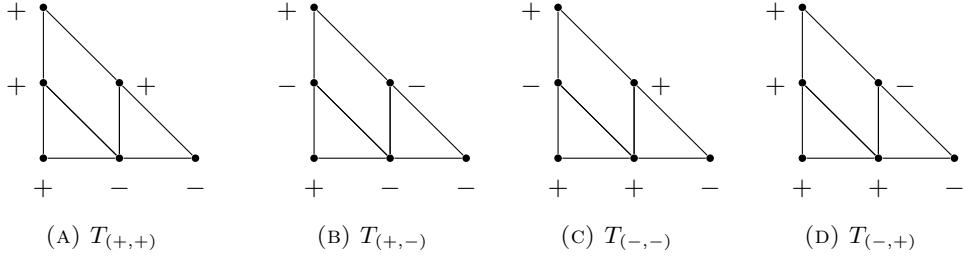
\begin{figure}
\subcaptionbox{$T_{(+,+)}$}{
\begin{tikzpicture}
\def\x{2}
\def\y{2}
\def\shift{20}
% marked
\coordinate[shape=circle,inner sep=1pt,fill,label={[yshift=-\shift pt]$+$}] (P1) at (\x,\y);
\coordinate[shape=circle,inner sep=1pt,fill,label={[yshift=-\shift pt]$-$}] (P2) at (\x+1,\y);
\coordinate[shape=circle,inner sep=1pt,fill,label={[xshift=-0.5*\shift pt,yshift=-0.5*\shift pt]$+$}] (P3) at (\x,\y+1);
\coordinate[shape=circle,inner sep=1pt,fill,label={[yshift=-\shift pt]$-$}] (P4) at (\x+2,\y);
\coordinate[shape=circle,inner sep=1pt,fill,label={[xshift=0.5*\shift pt,yshift=-0.5*\shift pt]$+$}] (P5) at (\x+1,\y+1);
\coordinate[shape=circle,inner sep=1pt,fill,label={[xshift=-0.5*\shift pt,yshift=-0.5*\shift pt]$+$}] (P6) at (\x,\y+2);
% lines
\path[draw] (P1)--(P2)--(P3)--(P1);
\path[draw] (P3)--(P6)--(P5)--(P2)--(P3);
\path[draw] (P2)--(P5)--(P4)--(P2);
\end{tikzpicture}
}\hspace{10pt}
\subcaptionbox{$T_{(+,-)}$}{
\begin{tikzpicture}
\def\x{2}
\def\y{2}
\def\shift{20}
% marked
\coordinate[shape=circle,inner sep=1pt,fill,label={[yshift=-\shift pt]$+$}] (P1) at (\x,\y);
\coordinate[shape=circle,inner sep=1pt,fill,label={[yshift=-\shift pt]$-$}] (P2) at (\x+1,\y);
\coordinate[shape=circle,inner sep=1pt,fill,label={[xshift=-0.5*\shift pt,yshift=-0.5*\shift pt]$-$}] (P3) at (\x,\y+1);
\coordinate[shape=circle,inner sep=1pt,fill,label={[yshift=-\shift pt]$-$}] (P4) at (\x+2,\y);
\coordinate[shape=circle,inner sep=1pt,fill,label={[xshift=0.5*\shift pt,yshift=-0.5*\shift pt]$-$}] (P5) at (\x+1,\y+1);
\coordinate[shape=circle,inner sep=1pt,fill,label={[xshift=-0.5*\shift pt,yshift=-0.5*\shift pt]$+$}] (P6) at (\x,\y+2);
% lines
\path[draw] (P1)--(P2)--(P3)--(P1);
\path[draw] (P3)--(P6)--(P5)--(P2)--(P3);
\path[draw] (P2)--(P5)--(P4)--(P2);
\end{tikzpicture}
}
\subcaptionbox{$T_{(-,-)}$}{
\begin{tikzpicture}
\def\x{2}
\def\y{2}
\def\shift{20}
% marked
\coordinate[shape=circle,inner sep=1pt,fill,label={[yshift=-\shift pt]$+$}] (P1) at (\x,\y);
\coordinate[shape=circle,inner sep=1pt,fill,label={[yshift=-\shift pt]$+$}] (P2) at (\x+1,\y);
\coordinate[shape=circle,inner sep=1pt,fill,label={[xshift=-0.5*\shift pt,yshift=-0.5*\shift pt]$-$}] (P3) at (\x,\y+1);
\coordinate[shape=circle,inner sep=1pt,fill,label={[yshift=-\shift pt]$-$}] (P4) at (\x+2,\y);
\coordinate[shape=circle,inner sep=1pt,fill,label={[xshift=0.5*\shift pt,yshift=-0.5*\shift pt]$+$}] (P5) at (\x+1,\y+1);
\coordinate[shape=circle,inner sep=1pt,fill,label={[xshift=-0.5*\shift pt,yshift=-0.5*\shift pt]$+$}] (P6) at (\x,\y+2);
% lines
\path[draw] (P1)--(P2)--(P3)--(P1);
\path[draw] (P3)--(P6)--(P5)--(P2)--(P3);
\path[draw] (P2)--(P5)--(P4)--(P2);
\end{tikzpicture}
}
\subcaptionbox{$T_{(-,+)}$}{
\begin{tikzpicture}
\def\x{2}
\def\y{2}
\def\shift{20}
% marked
\coordinate[shape=circle,inner sep=1pt,fill,label={[yshift=-\shift pt]$+$}] (P1) at (\x,\y);
\coordinate[shape=circle,inner sep=1pt,fill,label={[yshift=-\shift pt]$+$}] (P2) at (\x+1,\y);
\coordinate[shape=circle,inner sep=1pt,fill,label={[xshift=-0.5*\shift pt,yshift=-0.5*\shift pt]$+$}] (P3) at (\x,\y+1);
\coordinate[shape=circle,inner sep=1pt,fill,label={[yshift=-\shift pt]$-$}] (P4) at (\x+2,\y);
\coordinate[shape=circle,inner sep=1pt,fill,label={[xshift=0.5*\shift pt,yshift=-0.5*\shift pt]$-$}] (P5) at (\x+1,\y+1);
\coordinate[shape=circle,inner sep=1pt,fill,label={[xshift=-0.5*\shift pt,yshift=-0.5*\shift pt]$+$}] (P6) at (\x,\y+2);
% lines
\path[draw] (P1)--(P2)--(P3)--(P1);
\path[draw] (P3)--(P6)--(P5)--(P2)--(P3);
\path[draw] (P2)--(P5)--(P4)--(P2);
\end{tikzpicture}
}
\caption{Signed regular marked subdivisions $T_v$ with $v\in\{\pm\}^2$.}
\label{fig:signedregularmarkedsubdivisions}
\end{figure}
The real tropical hypersurface $C_f\subset\RealTropicalgroup$ is shown in \Cref{fig:realtropicalhypersurface}. It is worth mentioning that charts of $\RealTh{f}$ are not necessarily connected or balanced.
\begin{figure}
\subcaptionbox{$C_{f,(+,+)}$}{
\begin{tikzpicture}
\def\k{0.75}
\def\l{0.75}
\def\shift{20}
% marked
\coordinate[shape=circle,inner sep=1pt,fill,label={[xshift=-\shift pt,yshift=-3 pt]$(-1,-1)$}] (v1) at (-\k,-\k);
\coordinate[shape=circle,inner sep=1pt,fill,label={[xshift=-3 pt]$(0,0)$}] (v2) at (0,0);
\coordinate[shape=circle,inner sep=1pt,fill,label={[yshift=-\shift pt]$(1,0)$}] (v3) at (\k,0);
% lines
% \path[draw] (v1)--(-\k-\l,-\k);
\path[draw] (v1)--(-\k,-\k-\l);
\path[draw] (v1)--(v2);
% \path[draw] (v2)--(-\k-\l,0);
% \path[draw] (v2)--(\l,\l);
\path[draw] (v2)--(v3);
% \path[draw] (v3)--(\k,-\k-\l);
\path[draw] (v3)--(\k+\l,\l);
\end{tikzpicture}
}\hspace{10pt}
\subcaptionbox{$C_{f,(+,-)}$}{
\begin{tikzpicture}
\def\k{0.75}
\def\l{0.75}
\def\shift{20}
% marked
\coordinate[shape=circle,inner sep=1pt,fill,label={[xshift=\shift+3 pt,yshift=-8 pt]$(-1,-1)$}] (v1) at (-\k,-\k);
\coordinate[shape=circle,inner sep=1pt,fill,label={[xshift=-6 pt]$(0,0)$}] (v2) at (0,0);
% \coordinate[shape=circle,inner sep=1pt,fill,label={[yshift=-\shift pt]$(1,0)$}] (v3) at (\k,0);
% lines
\path[draw] (v1)--(-\k-\l,-\k);
\path[draw] (v1)--(-\k,-\k-\l);
% \path[draw] (v1)--(v2);
\path[draw] (v2)--(-\k-\l,0);
\path[draw] (v2)--(\l,\l);
% \path[draw] (v2)--(v3);
% \path[draw] (v3)--(\k,-\k-\l);
% \path[draw] (v3)--(\k+\l,\l);
\end{tikzpicture}
}
\subcaptionbox{$C_{f,(-,-)}$}{
\begin{tikzpicture}
\def\k{0.75}
\def\l{0.75}
\def\shift{20}
% marked
\coordinate[shape=circle,inner sep=1pt,fill,label={[yshift=-\shift pt]$(-1,-1)$}] (v1) at (-\k,-\k);
\coordinate[shape=circle,inner sep=1pt,fill,label={[]$(0,0)$}] (v2) at (0,0);
\coordinate[shape=circle,inner sep=1pt,fill,label={[yshift=-14 pt,xshift=14 pt]$(1,0)$}] (v3) at (\k,0);
% lines
\path[draw] (v1)--(-\k-\l,-\k);
% \path[draw] (v1)--(-\k,-\k-\l);
\path[draw] (v1)--(v2);
\path[draw] (v2)--(-\k-\l,0);
% \path[draw] (v2)--(\l,\l);
% \path[draw] (v2)--(v3);
\path[draw] (v3)--(\k,-\k-\l);
\path[draw] (v3)--(\k+\l,\l);
\end{tikzpicture}
}
\subcaptionbox{$C_{f,(-,+)}$}{
\begin{tikzpicture}
\def\k{0.75}
\def\l{0.75}
\def\shift{20}
% marked
% \coordinate[shape=circle,inner sep=1pt,fill,label={[yshift=-\shift pt]$(-1,-1)$}] (v1) at (-\k,-\k);
\coordinate[shape=circle,inner sep=1pt,fill,label={[yshift=-\shift pt]$(0,0)$}] (v2) at (0,0);
\coordinate[shape=circle,inner sep=1pt,fill,label={[xshift=14 pt,yshift=-10 pt]$(1,0)$}] (v3) at (\k,0);
% lines
% \path[draw] (v1)--(-\k-\l,-\k);
% \path[draw] (v1)--(-\k,-\k-\l);
% \path[draw] (v1)--(v2);
% \path[draw] (v2)--(-\k-\l,0);
\path[draw] (v2)--(\l,\l);
\path[draw] (v2)--(v3);
\path[draw] (v3)--(\k,-\k-\l);
% \path[draw] (v3)--(\k+\l,\l);
\end{tikzpicture}
}
\caption{Real plane tropical curve $\RealTh{f}$.}
\label{fig:realtropicalhypersurface}
\end{figure}
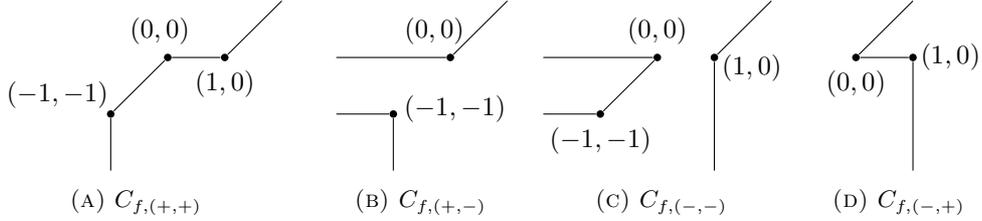
\end{example}

\begin{remark}[Duality in the real/complex case]\label{rema:dualitycomplexcomparedreal}
The duality in the complex case explained in \Cref{rema:duality} implies that we can deduce the type of the subdivision (cf. \Cref{defi:signedmarkedsubdivision}) of a given complex plane tropical curve. In the real case this is not true, e.g. there might be signed marked polytopes in the subdivision without any influence/evidence in the real plane tropical curve.
\end{remark}

%\begin{example}
%Beispiel von subdivisions mit dreieck/viereck und ``\#''-en an den knoten.
%\end{example}

\begin{remark}[Switching signs]\label{rema:switchingsigns}
As explained in \Cref{rema:equivalentsignedregularmarkedsubdivisions}, an element $(s,u)\in\signedsecondaryfan{\support}$ defines a polynomial $f=\bigoplus_{\alpha\in\support} (s_\alpha,u_\alpha)w^\alpha\in\RealTroplauringtwovars$ as well as a set of signed regular marked subdivisions $T_v$ with $v\in\pvector{2}$. Consider $(s',u)\in\signedsecondaryfan{\support}$ such that $s$ is sign equivalent to $s'$. It defines a real plane tropical curve $C_{f'}$ arising from the real tropical polynomial $f'=\bigoplus_{\alpha\in\support} (s_\alpha',u_\alpha)w^\alpha\in\RealTroplauringtwovars$. $C_f$ and $C_{f'}$ are related as follows: since $s'$ is sign equivalent to $s$ there is an element $v\in\pvector{2}$ such that $s'=\psi_\support(v)s$. Consequently, the signed regular marked subdivision $T'=T_{(+,+)}'$ defined by $(s',u)$ equals $T_v$, the signed regular marked subdivision obtained from $(s,u)$ with signs switched according to $v$. Since we index charts by elements of $\pvector{2}$ we conclude: the charts of $C_{f'}$ and $C_{f}$ are equal up to a permutation. This means that $T_{v'}'$ equals $T_{v\cdot v'}$ for all $v'\in\pvector{2}$.
\end{remark}

\begin{definition}[Signed secondary fan]\label{defi:signedsecondaryfan}
By abuse of notation, we write $s(w)$ for the sign vector of $w\in\signedsecondaryfan{\support}$ and $|w|$ for the modulus (cf. the notations defined for $\RealTropicalgroup$ in \Cref{subsec:realpuiseuxseries}). We define an equivalence relation on $\signedsecondaryfan{\support}$. Two elements $w,w'\in\signedsecondaryfan{\support}$ are called equivalent if and only if $s(w)$ is sign equivalent to $s(w')$ and $\sigma(|w|)=\sigma(|w'|)$. We call the set $\signedsecondaryfan{\support}$ with its decomposition into equivalence classes the \textit{signed secondary fan} of the point configuration $\support\subset\Z^2$.
\end{definition}

\begin{remark}[Equivalence classes of $\signedsecondaryfan{\support}$.]\label{rema:equivalenceclassessignedsecondaryfan}
Equivalence classes of regular marked subdivisions form cones $\sigma_T$ providing $\secondaryfan{\support}$ with a fan structure. As $\RealTropicalgroup$ has no reasonable addition (cf. \Cref{defi:realtropicalgroup}), the signed secondary fan $\signedsecondaryfan{\support}$ does not carry a fan structure. We denote the equivalence class of an element $(s,u)\in\signedsecondaryfan{\support}$ by $s\times\sigma_T$ where any $u\in\relint(\sigma_{T})$ provides the regular subdivision $T$. The representatives of $s\times\sigma_T$ form the set $G\cdot s \times \sigma_T$ where $G\cdot s$ denotes the orbit of $s$ under $G$. We define the dimension of $s\times \sigma_T$ by $\dim(\sigma_T)$.
% Set-theoretically, we have $\signedsecondaryfan{\support}=\RealTropicalgroup^m$. As $\RealTropicalgroup^{m}$ is a group 
\end{remark}

\begin{remark}[Lineality group]\label{rema:linealitygroup}
Consider the equivalence class $s\times\sigma_T\subset\signedsecondaryfan{\support}$. The secondary fan $\secondaryfan{\support}$ contains the lineality space $\rowspace(A')$, generated by $\ones{m}$ and the two vectors containing the $x$-coordinates and $y$-coordinates of the points in $\support$. Consequently, $s\times\sigma_T$ is invariant under $(+)^m\times\rowspace(A')$ since it does not change the signs or subdivision inappropriately. As the sign of an equivalence class is unique up to sign equivalence we see that $s\times\sigma_T$ is invariant under the lineality group $G\times\rowspace(A')$ (cf. \Cref{rema:equivalenceclassessignedsecondaryfan}).
% contains the set $(G\times\rowspace(A'))\odot_{\R}(s\times\zeros{m})$.
\end{remark}

\subsection{Real Tropicalization of $\Realsingsat{\ones{2}}$}\label{subsec:realsingsatone}

In this section we study the real tropicalization of $\Realsingsat{\ones{2}}=\Var{\Ideal{I}}$ with $\Ideal{I}\subset\Realcoefficientring{m}$ as defined in \Cref{nota:sec3} (in particular \Cref{eq:realidealsingsatone}). According to \Cref{eq:realshiftedsupportmatrix} we have
\begin{equation}\notag
\Realsingsat{\ones{2}}=\Var{\Ideal{I}}=\ker(A')~\With{}~A'=\begin{bmatrix}
\ones{m}^\top\\
A
\end{bmatrix}\in\Z^{3\times m}.
\end{equation}
Here, $A\in\Z^{2\times m}$ denotes the matrix representation of $\support$ and $A'$ its shift into $\R^{3}$. Let $\idealmatroid{\Ideal{I}}$ denote the matroid associated to the linear ideal $\Ideal{I}$. In particular, it is the matroid associated to the row space of $A'$. Its dual matroid is the vector matroid $\vectormatroid{A'}$ associated to the kernel of $A'$. It encodes the affine dependencies among the points of $\support$. Let $\{1,2,3\}$ be a basis of $\vectormatroid{A'}$ and, therefore, $\{\alpha_1,\alpha_2,\alpha_3\}$ is an affine basis of the affine hyperplane $\affinespan{\support'}\subset\R^{3}$. Hence, we obtain an equivalent coefficient matrix of $\Ideal{I}$ of the form
\begin{equation}
A'=\begin{bmatrix}
  \mathbb{E}_3 & \bar{A}'\label{eq:galedual}
  \end{bmatrix}
\end{equation}
such that $\bar{A}'\in\Q^{3\times m-3}$. Moreover, $\left( \bar{A}' \right)_j =\begin{bmatrix} a_j & b_j & c_j \end{bmatrix}^\top$ such that
\begin{equation}
\alpha_j =a_j\alpha_1 + b_j\alpha_2 + c_j\alpha_3~\And{}~a_{j}+b_{j}+c_{j}=1.\label{eq:affinebasis}
\end{equation}
for all $j$. Hence, column $(\bar{A}')_j$ with $j\in\{4,\ldots,m\}$ of $\bar{A}'$ contains the coordinates of $\alpha_j$ with respect to the affine basis $\{\alpha_1,\alpha_2,\alpha_3\}$.

\begin{remark}[Affine dependencies]\label{rema:planarcircuits}
Recall that $\vectormatroid{A'}$ denotes the vector matroid arising from the linear dependencies among the columns of $A'$. As $A'$ is the matrix representation of $\support'=\{1\}\times\support$ it encodes the information about affine dependencies among the point configuration $\support$. The fixed point configuration $\support$ lives in the plane. In \Cref{fig:planarcircuits} we list all planar circuits that may occur and we refer to them as circuits of type (A), (B) and (C).
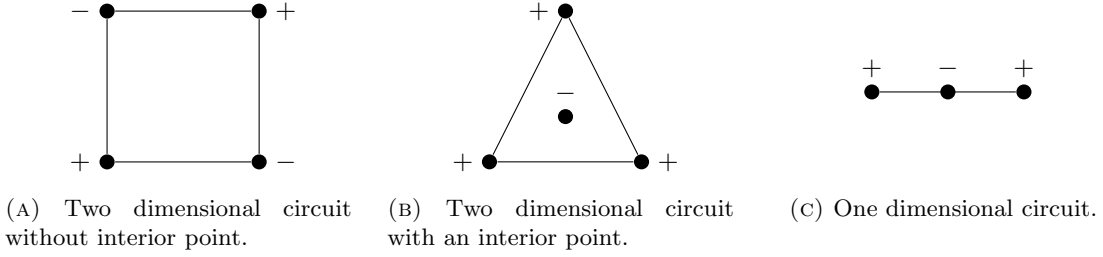
\begin{figure}[h]
\subcaptionbox{Two dimensional circuit without interior point.\label{fig:quadrangle}}[0.3\textwidth]{
\begin{tikzpicture}
\def\x{0}
\def\y{0}
\def\shift{1}
\coordinate[shape=circle,inner sep=2pt,fill,label={[xshift=10 pt,yshift=-10 pt]$-$}] (P1) at (\x+\shift,\y-\shift);
\coordinate[shape=circle,inner sep=2pt,fill,label={[xshift=10 pt,yshift=-10 pt]$+$}] (P2) at (\x+\shift,\y+\shift);
\coordinate[shape=circle,inner sep=2pt,fill,label={[xshift=-10 pt,yshift=-10 pt]$-$}] (P3) at (\x-\shift,\y+\shift);
\coordinate[shape=circle,inner sep=2pt,fill,label={[xshift=-10 pt,yshift=-10 pt]$+$}] (P4) at (\x-\shift,\y-\shift);
\path[draw,black] (P4) -- (P3) -- (P2) -- (P1) -- (P4);
\end{tikzpicture}
}\hspace{10pt}
\subcaptionbox{Two dimensional circuit with an interior point.\label{fig:triangle}}[0.3\textwidth]{
\begin{tikzpicture}
\def\x{0}
\def\y{0}
\def\shift{1}
\coordinate[shape=circle,inner sep=2pt,fill,label={[xshift=10 pt,yshift=-10 pt]$+$}] (P1) at (\x+\shift,\y);
\coordinate[shape=circle,inner sep=2pt,fill,label={[xshift=-10 pt,yshift=-10 pt]$+$}] (P2) at (\x-\shift,\y);
\coordinate[shape=circle,inner sep=2pt,fill,label={[yshift=-\shift pt]$-$}] (P3) at (\x,\y+0.6);
\coordinate[shape=circle,inner sep=2pt,fill,label={[xshift=-10 pt,yshift=-10 pt]$+$}] (P4) at (\x,\y+\shift+1);
\path[draw,black] (P4) -- (P1) -- (P2) -- (P4);
\end{tikzpicture}
}\hspace{10pt}
\subcaptionbox{One dimensional circuit.\label{fig:line}}[0.3\textwidth]{
\begin{tikzpicture}
\def\x{0}
\def\y{0}
\def\shift{1}
\coordinate[shape=circle] (P0) at (\x,-\shift);
\coordinate[shape=circle,inner sep=2pt,fill,label={[yshift=-\shift pt]$+$}] (P1) at (\x-\shift,\y);
\coordinate[shape=circle,inner sep=2pt,fill,label={[yshift=-\shift pt]$+$}] (P2) at (\x+\shift,\y);
\coordinate[shape=circle,inner sep=2pt,fill,label={[yshift=-\shift pt]$-$}] (P3) at (\x,\y);
\path[draw,black] (P1) -- (P2);
\end{tikzpicture}
}
\caption{Planar circuits together with signs that can be realized.}
\label{fig:planarcircuits}
\end{figure}
To obtain a top-dimensional circuit in $\R^{n}$, add a point to the $n$-dimensional standard simplex $\conv(0,e_{1},\ldots,e_{n})$ that is either properly contained in the simplex or outside the simplex. We can recover the signs of a signed circuit $C$ of $\vectormatroid{A'}$ from a single sign attached to a vertex of the unsigned circuit $\underline{C}$. If we equip a single arbitrary point of any of the sketched but unsigned circuits with a sign the remaining signs are uniquely determined by Radon's Theorem (see e.g. \cite[Chapter 2.1, Theorem 1.2]{gruber}). Hence, the sign distributions correspond uniquely to signed circuits of $\vectormatroid{A'}$.
\end{remark}

We conclude: $\Realsingsat{\ones{2}}=\Var{\Ideal{I}}=\ker(A')$ is a linear space. The oriented matroid $M$ associated to the linear ideal $\Ideal{I}$ describes the real tropicalization of $\ker(A')$ completely (cf. \Cref{theo:realtropicalbasis}), i.e. we have $\realtrop(\Var{\Ideal{I}})=\bigcap_{C\in\circuits} \RealTh{l_C}$ (cf. \Cref{rema:orientedmatroidassociatedtoideal} for notations). Moreover, \Cref{theo:realtropicallinearspace=signedbergmanfan} shows that the $s$-chart of $\realtrop(\Var{\Ideal{I}})$ equals $\signedbergmanfan{s}{M}$. In \Cref{sec:signedbergmanfans} we showed that
\begin{equation}\notag
\signedbergmanfan{s}{M}=\bigcup_{\substack{\flag\triangleleft\underline{M}:\\\flag~\text{is an}~s\text{-flag}}}\sigma_{\flag}.
\end{equation}
Hence, we can study $\realtrop(\Realsingsat{\ones{2}})$ by studying $s$-flags $\flag\triangleleft M$ of the associated oriented matroid to $\Ideal{I}$. In \cite{markwigmarkwigshustin} maximal flags of the unoriented matroid $\underline{M}$ were studied and classified. We enhance the classification to the oriented matroid $M$.

\begin{remark}[Gale duals]\label{rema:galedual3}
A Gale dual of a matrix $A$ is a matrix $G$ whose rows span the kernel of $A$. From the affine basis $\{\alpha_1,\alpha_2,\alpha_3\}$ and, therefore, $A'$ as shown in \Cref{eq:galedual} we get the Gale dual \[G=\begin{bmatrix}
  -(\bar{A}')^\top & \mathbb{E}_{m-3}
  \end{bmatrix}.\] of $A'$. In particular, $\vectormatroid{G}=\idealmatroid{\Ideal{I}}$, i.e. $G$ realizes $\idealmatroid{\Ideal{I}}$ as a vector matroid. We refer to the $i$-th column of $G$ by $g_i$ such that $G=\begin{bmatrix} g_1&\cdots&g_m\end{bmatrix}$. Note that $g_{i+3}=e_i$ for $1\leq i\leq m-3$. The first three columns $g_1,g_2,g_3$ contain the $x$,$y$ and $z$ coordinates of the point configuration $\support$ with respect to the basis $\{\alpha_1,\alpha_2,\alpha_3\}$. For example, $(g_1)_j=-a_j$ (cf. \Cref{eq:affinebasis}). We refer to the first three columns as \textit{special columns} of $G$.
\end{remark}

% \begin{definition}\label{defi:flatdifference}
% Let $s\in\topes$ be a sign vector and $\flag=(F_1,\ldots,F_k)\triangleleft M$ be an $s$-flag of flats of a matroid $M$. We define $F_{i,j}=F_i\setminus F_{j}$ for all $1\leq j\leq i\leq k$ where $F_0=\emptyset$ (cf. \Cref{defi:flagofsubsets}).
% \end{definition}

\begin{remark}[Maximal flags in $\umatroid{M}$.]\label{rema:classificationflagsunorientedmatroid}
In \cite[Lemma 3.7]{markwigmarkwigshustin} maximal flags of flats $\flag\triangleleft\umatroid{M}$ were classified. A flag of flats $\flag=(F_1,\ldots,F_{m-3})\triangleleft\umatroid{M}$ satisfies either
\begin{enumerate}
\item[(a)] $|F_{m-3,m-4}|=4$ and $|F_{j,j-1}|=1$ for all $j\neq m-3$, or
\item[(b)] $|F_{m-3,m-4}|=3$ and for some $k\neq m-3$ we have $|F_{k,k-1}|=2$ and $|F_{j,j-1}|=1$ for all $j\neq k,m-3$.
\end{enumerate}
In case (a) we have $F_{m-3,m-4}=\{i_1,i_2,i_3,i_4\}\subset\groundset{m}$ and any proper subset of $\{\alpha_{i_1},\alpha_{i_2},\alpha_{i_3},\alpha_{i_4}\}$ is affinely independent. Hence, $F_{m-4,m-3}$ is a (unsigned) circuit of type (A) or (B) (cf. \Cref{rema:planarcircuits}). In case (b) we have $F_{m-3,m-4}=\{i_1,i_2,i_3\}$ and $\{\alpha_{i_1},\alpha_{i_2},\alpha_{i_3}\}$ is affinely dependent. Thus, $F_{m-3,m-4}$ is a (unsigned) circuit of type (C). Moreover, all points $\alpha_r$ with $r\in F_{l,l-1}$ and $l>k$ are on the same line as $\{\alpha_{i_1},\alpha_{i_2},\alpha_{i_3}\}$.
\end{remark}

In the remaining part of this section we determine the sign vectors such that a given flag $\flag\triangleleft\underline{M}$ becomes an $s$-flag. Recall that this highly depends on the covectors as their zero-sets correspond to flats of $\underline{M}$. First, we like to know the set of covectors in $\cvector{M}$ morphing a flag $\flag\triangleleft\umatroid{M}$ of type (a) to a $s$-flag:

\begin{lemma}[$s$-flags of type (a)]\label{lemm:typeaflags}
Let $A'$ denote the matrix having the points of $\{1\}\times\support$ as columns, $G=\begin{bmatrix}g_1&\cdots&g_m\end{bmatrix}$ a Gale dual to $A'$ (cf. \Cref{rema:galedual3}) and let $M=\vectormatroid{G}$ denote the oriented vector matroid associated to $G$. Moreover, let $\vectormatroid{A'}$ be the oriented vector matroid associated to the shift of the point configuration $\support$ and let $\flag\triangleleft\underline{M}$ be a maximal flag of type (a) (cf. \Cref{rema:classificationflagsunorientedmatroid}) and $s\in\pvector{m}$ a pure sign vector. Then, $\flag$ is an $s$-flag if and only if there exists a signed circuit $C\in\circuits(\vectormatroid{A'})$ (with sign vector $s_C$ associated to $C$, cf. \Cref{rema:signvectorsandsignedsets}) such that $\underline{C}=F_{m-3,m-4}$ and $p_{\underline{C}}(s_C)= p_{\underline{C}}(s)$ (cf. \Cref{nota:sec3}), i.e. $F_{m-3,m-4}=C$ is a signed circuit  (cf. \Cref{rema:planarcircuits}).
\end{lemma}

\begin{proof}
The flag $\flag$ is a $s$-flag if there is a set of elements $y_1,\ldots,y_{m-3}\in\Dual{\R^{m-3}}$ providing the covectors $v_i=(\sign(\sk{y_i}{g_j}))_{j=1,\ldots,m}$ satisfying $v_i\subseteq s$ and $v_i^0 = F_i$ (cf. \Cref{rema:chaincovectors}). We like to work with a convenient Gale dual $G$ that allows an easy description of $\flag$. Therefore, we rename the elements of $E=\groundset{m}$ such that $F_{m-3,m-4}=\{a,b,c,i_{m-3}\}$ and $F_{j,j-1}=\{i_j\}$ for all $j\neq m-3$. Thus, $E=\{a,b,c,i_1,\ldots,i_{m-3}\}$ is the ground set, $B=\{i_1,\ldots,i_{m-3}\}$ is a basis of $\underline{M}$ and $\{\alpha_a,\alpha_b,\alpha_c\}$ is affinely independent since $F_{m-3,m-4}$ is a circuit. We consider $A'$ with resorted columns, i.e.
\begin{equation}
A'=\bordermatrix{
& a		& b		& c		& i_1		& \cdots& i_{m-3} 	\cr
& 1		& 1		& 1		& 1		& \cdots& 1 		\cr
& \alpha_a	& \alpha_b	& \alpha_c	& \alpha_{i_1} 	& \cdots&\alpha_{i_{m-3}}\cr
}.\label{eq:resortedA}
\end{equation}
Using \Cref{rema:galedual3} we work with a Gale dual of the form
\begin{equation}
G=\bordermatrix{
	& a		& b		& c		& i_1	& \cdots	& i_{m-3} 	\cr
i_1	& -a_{i_1}	& -b_{i_1}	& -c_{i_1}	& 1	& 		&		\cr
\vdots	& \vdots 	& \vdots 	& \vdots 	& 	& \ddots	&		\cr 
i_{m-3}	& -a_{i_{m-3}}	& -b_{i_{m-3}}	& -c_{i_{m-3}}	& 	&		& 1		\cr
}.\label{eq:Galeduala}
\end{equation}
Recall that the entries of $g_a,g_b,g_c$ in row $k$ are the negatives of the coordinates of $\alpha_k$ in the affine basis $\alpha_a,\alpha_b,\alpha_c$, in particular $1-a_{i_j}-b_{i_j}-c_{i_j}=0$ (cf. \Cref{eq:affinebasis}).

\bigskip

First, we show ``$\Rightarrow$'': suppose $\flag$ is an maximal $s$-flag. Thus, all flats $F_i$ are $s$-flats, i.e. we have a covector $v_i\in\cvector{\vectormatroid{G}}$ such that $v_i\subseteq s$ and $v_i^0=F_i$ for all $i=1,\ldots,m-3$. Let $y_i\in\Dual{\R^{m-3}}$ denote an element providing $v_i$, i.e. $v_i=(\sign(\sk{g_j}{y_i}))_{j\in E}$. Due to the shape of $G$ we have $g_{i_j}=e_j$ for $j=1,\ldots,m-3$. From this we can immediately conclude:
\begin{equation}
\forall l\in\{1,\ldots,m-3\},j\in\{1,\ldots,m-3\}: (v_l)_{i_j} =0\quad\Leftrightarrow\quad (y_l)_j = 0.\label{eq:signatunitvector} 
\end{equation}
Now, we focus on the covectors. To begin with, $v_{m-3}=\zeros{m}\in\cvector{\vectormatroid{G}}$ such that $v_{m-3}^0=F_{m-3}=E$. We have $F_{m-3,m-4}=\{a,b,c,i_{m-3}\}$ and, therefore, the covector $v_{m-4}$ satisfies $(v_{m-4})_j = 0 $ if and only if $j\neq a,b,c,i_{m-3}$ (cf. \Cref{rema:chaincovectors}). By definition, $v_{m-4}=(\sign(\sk{y_{m-4}}{g_j}))_{j\in E}$ where $y_{m-4}\in\Dual{\R^{m-3}}$. From \Cref{eq:signatunitvector} we conclude that $(y_{m-4})_j=0$ for all $j\in\{1,\ldots,m-4\}$. Thus, $y_{m-4}=t e_{m-3}$ and, therefore, we get
\begin{equation}\notag
v_{m-4}=(\sign(\sk{y_{m-4}}{g_j}))_{j\in E}=(\sign(\sk{te_{m-3}}{g_j}))_{j\in E}=(\sign(t(g_j)_{m-3}))_{j\in E}.
\end{equation}
Hence, the signs of $v_{m-4}$ are determined by the $i_{m-3}$-th row of $G$.
\begin{equation}
e_{m-3}\cdot G = \bordermatrix{
	& a		& b 		& c		& i_1	&\cdots& i_{m-4}& i_{m-3}\cr
	& -a_{i_{m-3}}	& -b_{i_{m-3}}	& -c_{i_{m-3}}	& 0 	&\cdots& 0	& 1\cr
}\label{eq:G_m-3}
\end{equation}
As outlined in the beginning, we have the affine relation $\alpha_{i_{m-3}} = a_{i_{m-3}}\alpha_a + b_{i_{m-3}}\alpha_b + c_{i_{m-3}} \alpha_c$ with $a_{i_{m-3}} + b_{i_{m-3}} + c_{i_{m-3}} = 1$. Thus, the non-zero entries of row $i_{m-3}$ of $G$ form the coefficients of a signed circuit $C\in\circuits(\vectormatroid{A'})$, i.e. $C^\pm=\{i\in\{a,b,c,i_{m-3}\}:(g_i)_{i_{m-3}}\gtrless 0\}$. However, all entries are non-zero as $a,b,c\notin F_j$ for $j\leq m-4$. This means, for example, \[(v_{m-4})_a = \sign(\sk{y_{m-4}}{g_a})=\sign(t(-a_{i_{m-3}}))=\sign(t)\sign(-a_{i_{m-3}}).\] Hence, $(v_{m-4})_j=\pm\sign(j)$ for all $j\in C$. Thus, the signs of the covector $v_{m-4}$ at $\{a,b,c,i_{m-3}\}$ coincide with the signs of a signed circuit. By assumption we have $v_{m-4}\subseteq s$ and, therefore, $t\gtrless 0$ such that $s_j = (v_{m-4})_j$ for all $j\in C$. Hence, $p_{\underline{C}}(s)=p_{\underline{C}}(s_C)$.\smallskip

Vice versa, suppose there is a signed circuit $C\in\circuits(\vectormatroid{A'})$ such that $s$ coincides with $s_C$ at $\underline{C}$, i.e. for all $i\in\underline{C}$ holds: $s_i=\pm$ if and only if $i\in C^\pm$. Now, we construct elements $y_i\in\Dual{\R^{m-3}}$ providing covectors $v_i=(\sign(\sk{y_i}{g_j}))_{j\in E}\in\cvector{\vectormatroid{G}}$ such that $v_i\subseteq s$ and $v_i^0=F_i$ for $i=0,\ldots,m-3$. 
We make use of the identical Gale dual $G$ constructed above. In detail, the $j$-th component of an element $y\in\Dual{\R^{m-3}}$ determines the sign of $v$ in component $i_j$ uniquely for $j\in\{1,\ldots,m-3\}$ as 
\begin{equation}\notag
(v)_{i_j}=\sign(\sk{y}{g_{i_j}})=\sign(\sk{y}{e_j})=\sign(y_j),
\end{equation}
cf. \Cref{eq:signatunitvector}. Consequently, by fixing the components individually, it is no problem to provide elements $y\in\Dual{\R^{m-3}}$ such that $v$ has desired signs in the components indexed by $\{i_1,\ldots,i_{m-3}\}$. Thus, for any choice of $s$ we can define elements $y\in\Dual{\R^{m-3}}$ such that $v=(\sign(\sk{y}{g_j}))_{j\in E}$ coincides with $s$ in the components $\{i_1,\ldots,i_{m-3}\}$. The major task is to guarantee that we can do this such that the signs of the remaining components indexed by $\underline{C}=\{a,b,c,i_{m-3}\}$ coincide with those of $s$. To begin with, note that $y_{m-3}=\zeros{m-3}$ such that $v_{m-3}^0=\zeros{m}$ in order to satisfy $v_{m-3}^0= F_{m-3}=E$. As above, row $i_{m-3}$ of $G$ contains the coefficients that provide the affine dependency 
\begin{equation}\notag
\alpha_{i_{m-3}}+(-a_{i_{m-3}})\alpha_a+(-b_{i_{m-3}})\alpha_b+(-c_{i_{m-3}})\alpha_c=0
\end{equation}
where $1-a_{i_{m-3}}-b_{i_{m-3}}-c_{i_{m-3}}=0$. For the signs at $\underline{C}$ we have \[s_a=\pm\sign(-a_{i_{m-3}}),~s_b=\pm\sign(-b_{i_{m-3}}),~s_c=\pm\sign(-c_{i_{m-3}})\And{}s_{i_{m-3}}=\pm\sign(1)\] as the circuit is unique up to a common factor $\pm 1$ and $p_{\underline{C}}(s)=p_{\underline{C}}(s_C)$. For $v_{m-4}$ it is required that $v_{m-4}^0=E\setminus F_{m-3,m-4}$. Thus, $(v_{m-4})_{i_j} =0$ for all $j\in\{1,\ldots,m-4\}$ and due to \Cref{eq:signatunitvector} this implies that $(y_{m-4})_j = 0 $ for all $j\in\{1,\ldots,m-4\}$. Hence, $y_{m-4}=t_{m-3} e_{m-3}$. Let us consider the sign at component $a$, i.e. \[(v_{m-4})_a=\sign(\sk{g_a}{y_{m-4}})=\sign(t_{m-3}(-a_{i_{m-3}}))=\pm s_a.\] As $v_{m-4}\subseteq s$ we pick $t_{m-3}\gtrless 0$ such that $(v_{m-4})_a=s_a$. This way we can specify the entries of $v_{m-4}$ at $\underline{C}$ conform to $s$, i.e. $(v_{m-4})_j=s_j$ for $j\in\underline{C}$. The next step is to define $v_{m-5},\ldots,v_{0}$ iteratively, based on $v_{m-4}$ as it is necessary to satisfy $v_{i}\subseteq v_{i-1}$ for $i=2,\ldots,m-3$ (cf. \Cref{rema:chaincovectors}). Note that $F_{j,j-1}=\{i_j\}$ for $j\neq m-3$, i.e. $v_{j}$ differs from $v_{j-1}$ in component $i_j$. Hence, $y_j$ differs from $y_{j-1}$ in component $j$. For example, $F_{m-4,m-5}=\{i_{m-4}\}$ such that we want $(v_{m-5})_{i_{m-4}}=s_{i_{m-4}}$ so that we define $y_{m-5}=y_{m-4} + t_{m-4}e_{m-4}$ with $t_{m-4}\gtrless 0$ according to $s_{i_{m-4}}$. In general, we define \[y_{m-k-1}=y_{m-k} + t_{m-k}e_{m-k}\] for $k=4,\ldots,m-2$ with $t_{m-k}\gtrless 0$ according to $s_{i_{m-k}}$. However, the specifications of $y_{m-k}$ at components $j\in\{1,\ldots,m-4\}$ may affect the signs of the components indexed by $\underline{C}$ of the corresponding covector. In detail, we have $(v_l)_{i_j}=\sign(t_j)$ for all $l\in\{1,\ldots,m-5\}$ and $l<j\leq m-3$ whereas e.g. $(v_l)_a = \sign\left(\sum_{j=3}^{s+1} t_{m-j}e_{m-j}(-a_{i_{m-j}})\right)$. Therefore, we stick to the following convention: we choose a very large $t_{m-3}\in\R$ such that $|t_{m-3}a_{i_{m-3}}|>\sum_{j=4}^{m-1} |t_{m-j}a_{i_{m-j}}|$, i.e. the signs at $\underline{C}$ remain fixed by the choice for $t_{m-3}$ and we get $v_{m-3}=\zeros{m}\subseteq v_{m-4}\subseteq\ldots\subseteq v_1\subseteq s$ with $v_i^0=F_i$ for all $i$.
\end{proof}

The situation for maximal flags $\flag\triangleleft\underline{M}$ of type (b) is quite similar but not identical. The proof of the previous lemma reveals that the signs at $F_{m-3,m-4}$ are related whereas all signs at $F_{j,j-1}$ with $j\neq m-3$ are unrelated. For a flag $\flag$ of type (b) we have $|F_{m-3,m-4}|=3$, $|F_{k,k-1}|=2$ for some $k\neq m-3$ and $|F_{j,j-1}|=1$ for all $j\neq k,m-3$. As one might expected the signs in $F_{m-3,m-4}$ and $F_{k,k-1}$ are related:
\begin{lemma}[$s$-flags of type (b)]\label{lemm:typebflags}
Let $A'$ denote the matrix having the points of $\{1\}\times\support$ as columns, $G=\begin{bmatrix}g_1&\cdots&g_m\end{bmatrix}$ a Gale dual to $A'$ (cf. \Cref{rema:galedual3}) and let $M=\vectormatroid{G}$ denote the oriented matroid associated to $G$. Moreover, let $\vectormatroid{A'}$ be the oriented vector matroid associated to the shift of the point configuration $\support$ and let $\flag\triangleleft\underline{M}$ be a maximal flag of type (b) (cf. \Cref{rema:classificationflagsunorientedmatroid}), i.e. $|F_{m-3,m-4}|=3$ and there is some $k\neq m-3$ such that $|F_{k,k-1}|=2$ and $|F_{j,j-1}|=1$ for all $j\neq k,m-3$ and let $s\in\pvector{m}$ be a pure sign vector. Recall that the affine span of the points in $\support$ corresponding to elements of $F_{m-3,m-4}$ forms a line $L$. Then, $\flag$ is an $s$-flag if and only if there exists a signed circuit $C\in\circuits(\vectormatroid{A'})$ (with sign vector $s_{C}$ associated to $C$, cf. \Cref{rema:signvectorsandsignedsets}) such that $\underline{C}=F_{m-3,m-4}$ and $p_{\underline{C}}(s)=p_{\underline{C}}(s_C)$ (cf. \Cref{nota:sec3}), i.e. $C=F_{m-3,m-4}$ is a signed circuit (cf. \Cref{rema:planarcircuits}), and, if $F_{k,k-1}=\{a,b\}$, $s_a = \pm s_b$ depending on whether $\alpha_a$ and $\alpha_b$ are on different/the same side of $L$.
\end{lemma}

\begin{proof}
Similarly to the proof of \Cref{lemm:typeaflags} we begin with specifying a Gale dual $G$ that is convenient for $\flag$. We reindex such that $F_{m-3,m-4}=\{i_{m-3},a,b\}$, $F_{k,k-1}=\{i_k,c\}$ and $F_{j,j-1}=\{i_j\}$ for all $j\neq k,m-3$. Thus, $E=\{a,b,c,i_1,\ldots,i_{m-3}\}$ and the three elements $a,b,c$ provide the affinely independent set $\{\alpha_a,\alpha_b,\alpha_c\}$. This is true since $\alpha_a,\alpha_b$ are affinely independent as $F_{m-3,m-4}$ is a circuit and $\alpha_c$ is not contained in the line $L$ containing the circuit $F_{m-3,m-4}$. Moreover, $B=\{i_1,\ldots,i_{m-3}\}$ forms a basis of $\underline{M}$. We consider $A'$ with re-sorted columns (cf. \Cref{eq:resortedA}). The first three special columns of the Gale dual $G$ obtained from $A'$ using the construction in \Cref{rema:galedual3} contain the negative coordinates of $\alpha_{i_j}$ with respect to the basis $\{\alpha_a,\alpha_b,\alpha_c\}$ for $j\in\groundset{m-3}$ (cf. \Cref{eq:Galeduala}).\smallskip

First, we show ``$\Rightarrow$'': suppose $\flag$ is a maximal $s$-flag, i.e. all flats $F_i$ are $s$-flats and there is a covector $v_i\in\cvector{\vectormatroid{G}}$ such that $v_i\subseteq s$ and $v_i^0=F_i$ for all $i\in E=\groundset{m-3}$. In particular, $v_{m-3}=\zeros{m}\in\cvector{\vectormatroid{G}}$ as we require $v_{m-3}^0=F_{m-3}=E$. We have $F_{m-3,m-4}=\{i_{m-3},a,b\}$ and, therefore, the element $v_{m-4}$ satisfies $(v_{m-4})_j = 0 $ if and only if $j\in E\setminus\{i_{m-3},a,b\}$. Let $y_{m-4}\in\Dual{\R^{m-3}}$ be an element such that $v_{m-4}=(\sign(\sk{y_{m-4}}{g_j}))_{j\in E}$. Recall that we have $g_{i_j}=e_{j}$ for $j\in\{1,\ldots,m-3\}$ such that 
\begin{equation}\notag
(v_{m-4})_{i_j} = \sign(\sk{y_{m-4}}{g_{i_j}})=\sign(\sk{y_{m-4}}{e_j})=\sign((y_{m-4})_{j})
\end{equation}
for $j\in\{1,\ldots,m-3\}$. Hence, $y_{m-4}=(0,\ldots,0,t)=te_{m-3}\in\Dual{\R^{m-3}}$ for some $t\neq 0$. Now, consider the row of $G$ indexed by $i_{m-3}$ (cf. \Cref{eq:G_m-3}). We have $(g_a)_{i_{m-3}}=-a_{i_{m-3}}$ and $(g_b)_{i_{m-3}}=-b_{i_{m-3}}$ whereas $(g_c)_{i_{m-3}}=-c_{i_{m-3}}=0$ as $F_{m-3,m-4}=\{i_{m-3},a,b\}$ is a circuit. According to \Cref{eq:affinebasis}, these coordinates provide the affine relation $\alpha_{i_{m-3}}=a_{i_{m-3}}\alpha_1+b_{i_{m-3}}\alpha_2$ (or, equivalently, $\alpha_{i_{m-3}}+(-a_{i_{m-3}})\alpha_1+(-b_{i_{m-3}})\alpha_2=0$) with $1-a_{i_{m-3}}-b_{i_{m-3}}-c_{i_{m-3}}=0$. Since we have $(g_{i_{m-3}})_{i_{m-3}}=1$, the $i_{m-3}$-th row of $G$ contains the coefficients of a signed circuit $C$ such that $\underline{C}=F_{m-3,m-4}$. Recall that $(v_{m-4})_j\neq 0$ if and only if $j=a,b,i_{m-3}$. We know that $y_{m-4}=te_{i_{m-3}}$. Thus, for example, 
\begin{equation}\notag
(v_{m-4})_a=\sign(\sk{y_{m-4}}{g_a})=\sign(t(-a_{i_{m-3}}))=\sign(t)\sign(-a_{i_{m-3}}).
\end{equation}
Hence, $(v_{m-4})_j=\pm\sign(j)$ for $j\in C$. Since $v_{m-4}\subseteq s$ we have $t\gtrless 0$ such that $s_j=(v_{m-4})_j$ for $j=a,b,i_{i_{m-3}}$, i.e. there is a signed circuit $C\in\circuits(\vectormatroid{A'})$ such that $\underline{C}=F_{m-3,m-4}$ and $p_{\underline{C}}(s)=p_{\underline{C}}(s_C)$. For $m-4\geq j\geq k+1$ we have $F_{j,j-1}=i_j$ and $\alpha_{i_j}$ is contained in the affine line spanned by $\{\alpha_a,\alpha_b\}$. Hence, we have $(g_c)_{i_j}=-c_{i_j}=0$ for all $j$ satisfying $k+1\leq j\leq m-3$ (cf. \Cref{eq:Galeduala}). Since $F_{k,k-1}=\{i_k,c\}$ we have $(g_c)_{i_k}=-c_{i_k}\neq 0$ --- otherwise, $c\notin F_{k,k-1}$. 
\begin{equation}
G=\bordermatrix{
	& a		& b		& c		& i_1	&\cdots &i_{k-1}&i_k	&i_{k+1}& \cdots&i_{m-3}\cr
i_1	& -a_{i_1}	& -b_{i_1}	& -c_{i_1}	& 1	& 	&	&	&	&	&	\cr
\vdots	& \vdots 	& \vdots 	& \vdots 	& 	& \ddots&	&	&	&	&	\cr
i_{k-1}	& -a_{i_{k-1}}	& -b_{i_{k-1}}	& -c_{i_{k-1}}	& 	&	& 1	&	&	&	&	\cr
i_k	& -a_{i_k}	& -b_{i_k}	& -c_{i_k}\neq0	&	&	&	& 1	&	&	&	\cr
i_{k+1}	& -a_{i_{k+1}}	& -b_{i_{k+1}}	& 0		&	&	&	&	&1	&	&	\cr
\vdots	& \vdots 	& \vdots 	& \vdots 	&	&	&	&	&	&\ddots	&	\cr
i_{m-3}	& -a_{i_{m-3}}	& -b_{i_{m-3}}	& 0		& 	&	&	&	&	&	& 1	\cr
}.\label{eq:Galedualb}
\end{equation}

Moreover, $(v_{k-1})_{i_{k}}\neq 0$ and consequently $(y_{k-1})_k \neq 0$. This implies $(v_{k-1})_c \neq 0$ and $(v_{k-1})_{i_k}\neq 0$. Recall that row $i_k$ of $G$ contains the coordinates of $\alpha_{i_k}$ with respect to the affine basis $\{\alpha_a,\alpha_b,\alpha_c\}$. Hence, the indices of the non-zero entries of the $i_k$-th row of $G$ form a circuit $C'\in\circuits(\vectormatroid{A'})$. Due to the reindexing we have $F_{k-1}=\{i_1,\ldots,i_{k-1}\}$, i.e. $(y_{k-1})_j=0$ for all $j\in\{1,\ldots,k-1\}$. By assumption, $v_{k-1}\subseteq s$, i.e. we have $(y_{k-1})_k\gtrless 0$ such that \[(v_{k-1})_{i_k}=\sign(\sk{y_{k-1}}{g_{i_k}})=\sign(\sk{y_{k-1}}{e_k})=\sign((y_{k-1})_k)=s_{i_k}.\] However, this implies $(v_{k-1})_{c} = \sign(\sk{y_{k-1}}{g_c})=\sign((y_{k-1})_k(-c_{i_k})) = s_c$, i.e. we get the condition $s_c = \sign(-c_{i_k}) s_{i_k}$ on the signs at $i_k$ and $c$ in $s$. The sign of $(-c_{i_k})$ purely depends on the relative position of $\alpha_{i_k}$ to $\alpha_a,\alpha_b$ and $\alpha_c$. Let $L_{ab}$ denote the affine line through $\alpha_a$ and $\alpha_b$.

\begin{itemize}
\item Let $\alpha_c$ and $\alpha_{i_k}$ be separated by $L_{ab}$ (cf. \Cref{fig:Lseparates}). Suppose $\alpha_a,\alpha_b,\alpha_c,\alpha_{i_k}$ form a circuit of type (A). Then, $\sign(-c_{i_k})=+$ since $\sign(1)=+$, as $1$ is the coefficient of $\alpha_{i_k}$ in the affine relation. If $\alpha_a,\alpha_b,\alpha_c,\alpha_{i_k}$ form a circuit of type (B) then $\sign(-c_{i_k})=+$ as well since $\sign(1)=+$ and $L_{ab}$ separates $\alpha_c$ and $\alpha_{i_k}$. If $\alpha_a,\alpha_b,\alpha_c,\alpha_{i_k}$ is a circuit of type (C) then (w.l.o.g.) $-a_{i_k}=0$ and due to $\sign(1)=+$ we have $\sign(-c_{i_k})=+$. We conclude: if $\alpha_c$ and $\alpha_{i_k}$ are separated by $L_{ab}$ then $s_{i_k}=s_c$.
\item Assume that $\alpha_c$ and $\alpha_{i_k}$ are on the same side of $L_{ab}$ (cf. \Cref{fig:Ldoesnotseparate}). If $\alpha_a,\alpha_b,\alpha_c,\alpha_{i_k}$ form a circuit of type (A) then $\sign(-c_{i_k})=-$ since the sign of the coefficient of $\alpha_{i_k}$, 1, is $+$.  If $\alpha_a,\alpha_b,\alpha_c,\alpha_{i_k}$ form a circuit of type (B) then either $\alpha_c$ or $\alpha_{i_k}$ is the interior point. We conclude that their signs differ, and since $\sign(1)=+$ we have $\sign(-c_{i_k})=-$. If $\alpha_a,\alpha_b,\alpha_c,\alpha_{i_k}$ form a circuit of type (C) we have (w.l.o.g.) $-a_{i_k}=0$. Then, $\alpha_c$ or $\alpha_{i_k}$ is the interior point --- otherwise $L_{ab}$ would separate $\alpha_c$ and $\alpha_{i_k}$. However, this implies $\sign(-c_{i_k})\neq\sign(1)=+$. We conclude: if $\alpha_c$ and $\alpha_{i_k}$ are on the same side of $L_{ab}$ then $s_{i_k}\neq s_c$.
\end{itemize}
Consequently, we have $s_c=s_{i_k}$ if and only if $L_{ab}$ separates $\alpha_c$ and $\alpha_{i_k}$.
\smallskip

For the other direction we use identical notations. Suppose we have a circuit $C\in\circuits(\vectormatroid{A'})$ such that $p_{F_{m-3,m-4}}(s)=s_C$ and $s_c=\pm s_{i_k}$, depending on the relative position of $\alpha_c$ and $\alpha_{i_k}$ to $L_{ab}$. The goal is to construct elements $y_i\in\Dual{\R^m}$ such that $v_i=(\sign(\sk{y_i}{g_j}))_{j\in E}$ for $i=1,\ldots,m-3$ satisfying $v_i^0=F_i$ and $v_i\subseteq s$. We begin with $v_{m-3}$. We require $v_{m-3}^0=F_{m-3}=E$, i.e. $v_{m-3}=\zeros{m}$. Recall from the proof of \Cref{lemm:typeaflags} that, for $j\in\{1,\ldots,m-3\}$, the $j$-th component of $y\in\Dual{\R^{m-3}}$ determines the sign of $v=(\sign(\sk{g_j}{y}))_{j\in E}$ in component $i_j$ as $g_{i_j}=e_j$. Consequently, it is no problem to provide elements $y\in\Dual{\R^{m-3}}$ such that the corresponding covectors have desired signs at components indexed by $\{i_1,\ldots,i_{m-3}\}$. It remains to show that the signs at $\{a,b,c\}$ remain unchanged, i.e. as initially required by the assumption.

Due to \Cref{eq:signatunitvector} we have $g_{i_j}=e_j$, i.e. $y_{m-3}=\zeros{m-3}$ provides $v_{m-3}$. We look for $v_{m-4}$ satisfying $v_{m-4}^0=F_{m-4}$. Note that $F_{m-3,m-4}=\{a,b,i_{m-3}\}$, i.e. we want $(v_{m-4})_j\neq 0$ if and only if $j\in\{a,b,i_{m-3}\}$. Equivalently, $(v_{m-4})_j=0$ if and only if $j\in\{i_1,\ldots,i_{m-4},c\}$. As $(y_{m-4})_j$ determines $(v_{m-4})_{i_j}$ we conclude that $(y_{m-4})_j=0$ for all $j\in\{1,\ldots,m-4\}$. Hence, $y_{m-4}=t_{m-3}e_{m-3}$ with $t_{m-3}\neq 0$. Recall that row $i_{m-3}$ of $G$ contains the negative coefficients of the affine relation $a_{i_{m-3}}\alpha_a+b_{i_{m-3}}\alpha_b = \alpha_{i_{m-3}}$. The indices form the signed circuit $C$. As $p_{\underline{C}}(s)=p_{\underline{C}}(s_C)$ we have $s_a=\pm\sign(-a_{i_{m-3}}),s_b=\pm\sign(-b_{i_{m-3}})$ and $s_{i_{m-3}}=\pm\sign(1)$ as the circuit is unique up to a common factor $\pm1$. We have seen that $y_{m-4}=t_{m-3}e_{m-3}$. We have $t_{m-3}\gtrless 0$ such that
\begin{equation}
(v_{m-4})_a=\sign(\sk{g_a}{y_{m-4}})=\sign(\sk{g_a}{t_{m-3}e_{m-3}})=\sign(t_{m-3}(g_a)_{m-3})=s_a.
\end{equation}
Then, $(v_{m-4})_j=\sign(j)$ for all $j\in C$, i.e. $v_{m-4}\subseteq s$. We define $v_{m-5},\ldots,v_{0}$ iteratively, based on $y_{m-4}$ (cf. \Cref{rema:chaincovectors}). Due to the reindexing we have $F_{j,j-1}=\{i_j\}$ for all $j\neq k,m-3$. We define $y_{m-j-1}=y_{m-j} + t_{m-j}e_{m-j}$ iteratively for $j=4,\ldots,m-1$ with $t_{m-j}\gtrless 0$ according to $s_{i_{m-j}}$. This way we can specify all signs according to those of $s$ at $\{i_1,\ldots,i_{m-4}\}$. The signs at $a,b,i_{m-3}$ were fixed with $y_{m-4}$. It remains to show what happens to the sign at $c$. This is determined in $v_{k-1}$: $v_k$ and $v_{k-1}$ differ at two indices, $i_k$ and $c$. Recall that $y_{k-1}=y_k+t_ke_k$. We pick $t_k\gtrless 0$ such that $(v_{k-1})_{i_k}=s_{i_k}$. Also note that $(g_c)_j=0$ for all $j\in\{i_{k+1},\ldots,i_{m-3}\}$, i.e. $(v_{l})_c =0$ for all $l\in\{k,\ldots,m-3\}$. Now, as $(y_{k-1})_k\neq 0$ we have $(v_{k-1})_c =\sign(\sk{g_c}{y_{k-1}})=\sign((g_c)_k t_k)\neq 0$ because $(g_c)_k\neq 0$. This is true since $\alpha_{i_k}$ is not contained in $L_{ab}$. However, we have $(g_c)_k=-c_{i_k}$. Moreover, row $i_k$ of $G$ contains the coefficients of the affine relation between $\alpha_{i_k},\alpha_a,\alpha_b$ and $\alpha_c$. As $\sign(i_k)=+$ we can determine the sign of $c$:
\begin{itemize}
\item Suppose $L_{ab}$ separates $\alpha_c$ and $\alpha_{i_k}$ (cf. \Cref{fig:Lseparates}). If $\alpha_a,\alpha_b,\alpha_c$ and $\alpha_{i_k}$ form a circuit of type (A) then $\sign(c)=+$. If it is a circuit of type (B) then either $\alpha_a$ or $\alpha_b$ is the interior point. Hence, $\sign(c)=+$. If it is a circuit of type (C) we have (w.l.o.g.) $a_{i_k}=0$, i.e. $\alpha_c$ is on a line with $\alpha_{i_k}$ and $\alpha_b$. Consequently, $\sign(c)=\sign(i_k)=+$.
\item Suppose $L_{ab}$ does not separate $\alpha_c$ and $\alpha_{i_k}$ (cf. \Cref{fig:Ldoesnotseparate}). If $\alpha_a,\alpha_b,\alpha_c$ and $\alpha_{i_k}$ form a circuit of type (A) we have $\sign(c)=-\neq\sign(i_k)=+$. If it is a circuit of type (B) then either $\alpha_{i_k}$ or $\alpha_c$ is the interior point --- their signs are not equal, i.e. $\sign(c)=-$. If it is a circuit of type (C) we have (w.l.o.g.) $a_{i_k}=0$, i.e. $\alpha_c$, $\alpha_{i_k}$ and $\alpha_b$ are on a line. Thus, (w.l.o.g.) $\alpha_c$ is the interior point and $\alpha_c$ and $\alpha_{i_k}$ are the boundary points such that $\sign(c)=-$ since $\sign(i_k)=+$.
\end{itemize}
We conclude that we get a sign at $c$ and $i_k$ in $v_{k-1}$, depending on the affine relations explained above. By assumption $s_c$ and $s_{i_k}$ satisfy the identical condition, i.e. we can pick $t_k\gtrless 0$ such that $s_{i_k}=(v_{k-1})_{i_k}$ and $s_c=(v_{k-1})_c$ in order to satisfy $v_{k-1}\subseteq s$. In order to guarantee that all subsequent covectors $v_j$ with $j\leq k-2$ stick to these signs at $c$ and $i_k$ we pick $t_{k}$ big enough (cf. discussions about $t_{m-3}$ before).
\end{proof}

\begin{figure}
\centering
\subcaptionbox{$s=(+,-,-,+,-,+)$.}[0.3\textwidth]{
\begin{tikzpicture}
\def\x{0}
\def\y{0}
\def\shift{1}
\coordinate[shape=circle,inner sep=2pt,fill,label={[yshift=-20 pt]$+$}] (P1) at (\x,\y);
\coordinate[shape=circle,inner sep=2pt,fill,label={[yshift=-20 pt]$-$}] (P2) at (\x+\shift,\y);
\coordinate[shape=circle,inner sep=2pt,fill,label={[xshift=-10 pt]$-$}] (P3) at (\x,\y+\shift);
\coordinate[shape=circle,inner sep=2pt,fill,label={[yshift=-20 pt]$+$}] (P4) at (\x+2*\shift,\y);
\coordinate[shape=circle,inner sep=2pt,fill,label={[yshift=-\shift pt]$-$}] (P5) at (\x+\shift,\y+\shift);
\coordinate[shape=circle,inner sep=2pt,fill,label={[yshift=-\shift pt]$+$}] (P6) at (\x,\y+2*\shift);
\path[draw,black] (P4) -- (P5) -- (P2) -- (P4);
\path[draw,black] (P1) -- (P2) -- (P5) -- (P6) -- (P3) -- (P1);
\end{tikzpicture}
}
\hspace{5pt}
\subcaptionbox{$s=(+,-,-,+,+,-)$}[0.3\textwidth]{
\begin{tikzpicture}
\def\x{0}
\def\y{0}
\def\shift{1}
\coordinate[shape=circle,inner sep=2pt,fill,label={[yshift=-20 pt]$+$}] (P1) at (\x,\y);
\coordinate[shape=circle,inner sep=2pt,fill,label={[yshift=-20 pt]$-$}] (P2) at (\x+\shift,\y);
\coordinate[shape=circle,inner sep=2pt,fill,label={[xshift=-10 pt]$-$}] (P3) at (\x,\y+\shift);
\coordinate[shape=circle,inner sep=2pt,fill,label={[yshift=-20 pt]$+$}] (P4) at (\x+2*\shift,\y);
\coordinate[shape=circle,inner sep=2pt,fill,label={[yshift=-\shift pt]$+$}] (P5) at (\x+\shift,\y+\shift);
\coordinate[shape=circle,inner sep=2pt,fill,label={[yshift=-\shift pt]$-$}] (P6) at (\x,\y+2*\shift);
\path[draw,black] (P4) -- (P5) -- (P2) -- (P4);
\path[draw,black] (P1) -- (P2) -- (P5) -- (P6) -- (P3) -- (P1);
\end{tikzpicture}
}
\hspace{5pt}
\subcaptionbox{$s=(+,-,-,-,+,+)$}[0.3\textwidth]{
\begin{tikzpicture}
\def\x{0}
\def\y{0}
\def\shift{1}
\coordinate[shape=circle,inner sep=2pt,fill,label={[yshift=-20 pt]$+$}] (P1) at (\x,\y);
\coordinate[shape=circle,inner sep=2pt,fill,label={[yshift=-20 pt]$-$}] (P2) at (\x+\shift,\y);
\coordinate[shape=circle,inner sep=2pt,fill,label={[xshift=-10 pt]$-$}] (P3) at (\x,\y+\shift);
\coordinate[shape=circle,inner sep=2pt,fill,label={[yshift=-20 pt]$-$}] (P4) at (\x+2*\shift,\y);
\coordinate[shape=circle,inner sep=2pt,fill,label={[yshift=-\shift pt]$+$}] (P5) at (\x+\shift,\y+\shift);
\coordinate[shape=circle,inner sep=2pt,fill,label={[yshift=-\shift pt]$+$}] (P6) at (\x,\y+2*\shift);
\path[draw,black] (P4) -- (P5) -- (P2) -- (P4);
\path[draw,black] (P1) -- (P2) -- (P5) -- (P6) -- (P3) -- (P1);
\end{tikzpicture}
}
\caption{Sign distributions for the regular marked subdivision of $\Delta=\conv(0,2e_1,2e_2)$.\label{fig:signdistributionflagoftypeb}}
\end{figure}
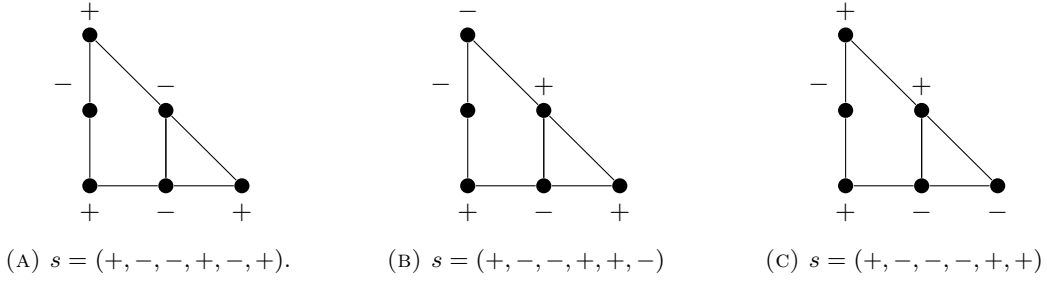

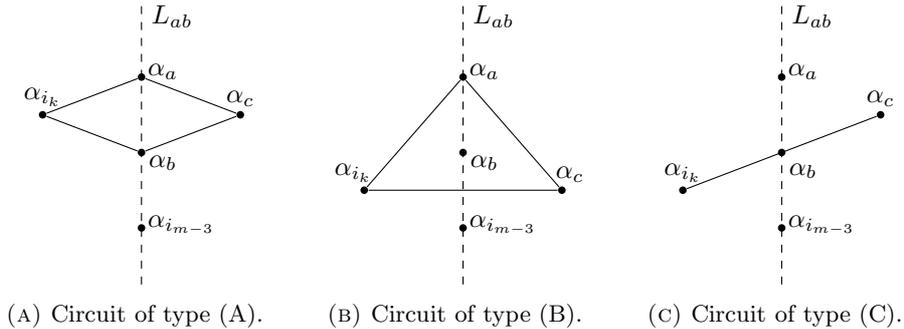
\begin{figure}
\subcaptionbox{Circuit of type (A).}{
\begin{tikzpicture}
\def\x{2}
\def\y{2}
\def\shift{20}
% marked
\coordinate[shape=circle,inner sep=1pt,fill,label={[xshift=8 pt,yshift=-6 pt]$\alpha_a$}] (P1) at (\x,\y+2);
\coordinate[shape=circle,inner sep=1pt,fill,label={[xshift=8 pt,yshift=-12 pt]$\alpha_b$}] (P2) at (\x,\y+1);
\coordinate[shape=circle,inner sep=1pt,fill,label={[xshift=15 pt,yshift=-8 pt]$\alpha_{i_{m-3}}$}] (P3) at (\x,\y);
\coordinate[shape=circle,inner sep=1pt,fill,label={[yshift=-2 pt]$\alpha_c$}] (P4) at (\x+1.3,\y+1.5);
\coordinate[shape=circle,inner sep=1pt,fill,label={[yshift=-2 pt]$\alpha_{i_k}$}] (P5) at (\x-1.3,\y+1.5);
\node at (\x+0.4,\y+2.8) {$L_{ab}$};
% lines
\path[draw,dashed] (\x,\y-0.75) -- (\x,\y+3);
\path[draw,black] (P1) -- (P4) -- (P2) -- (P5) -- (P1);
\end{tikzpicture}
}
\hspace{10pt}
\subcaptionbox{Circuit of type (B).}{
\begin{tikzpicture}
\def\x{2}
\def\y{2}
\def\shift{20}
% marked
\coordinate[shape=circle,inner sep=1pt,fill,label={[xshift=8 pt,yshift=-6 pt]$\alpha_a$}] (P1) at (\x,\y+2);
\coordinate[shape=circle,inner sep=1pt,fill,label={[xshift=8 pt,yshift=-12 pt]$\alpha_b$}] (P2) at (\x,\y+1);
\coordinate[shape=circle,inner sep=1pt,fill,label={[xshift=15 pt,yshift=-8 pt]$\alpha_{i_{m-3}}$}] (P3) at (\x,\y);
\coordinate[shape=circle,inner sep=1pt,fill,label={[yshift=-2 pt,xshift=3 pt]$\alpha_c$}] (P4) at (\x+1.3,\y+0.5);
\coordinate[shape=circle,inner sep=1pt,fill,label={[yshift=-2 pt,xshift=-4 pt]$\alpha_{i_k}$}] (P5) at (\x-1.3,\y+0.5);
\node at (\x+0.4,\y+2.8) {$L_{ab}$};
% lines
\path[draw,dashed] (\x,\y-0.75) -- (\x,\y+3);
\path[draw,black] (P1) -- (P4) -- (P5) -- (P1);
\end{tikzpicture}
}
\hspace{10pt}
\subcaptionbox{Circuit of type (C).}{
\begin{tikzpicture}
\def\x{2}
\def\y{2}
\def\shift{20}
% marked
\coordinate[shape=circle,inner sep=1pt,fill,label={[xshift=8 pt,yshift=-6 pt]$\alpha_a$}] (P1) at (\x,\y+2);
\coordinate[shape=circle,inner sep=1pt,fill,label={[xshift=8 pt,yshift=-15 pt]$\alpha_b$}] (P2) at (\x,\y+1);
\coordinate[shape=circle,inner sep=1pt,fill,label={[xshift=15 pt,yshift=-8 pt]$\alpha_{i_{m-3}}$}] (P3) at (\x,\y);
\coordinate[shape=circle,inner sep=1pt,fill,label={[yshift=-2 pt]$\alpha_c$}] (P4) at (\x+1.3,\y+1.5);
\coordinate[shape=circle,inner sep=1pt,fill,label={[yshift=-2 pt]$\alpha_{i_k}$}] (P5) at (\x-1.3,\y+0.5);
\node at (\x+0.4,\y+2.8) {$L_{ab}$};
% lines
\path[draw,dashed] (\x,\y-0.75) -- (\x,\y+3);
\path[draw,black] (P4) -- (P5);
\end{tikzpicture}
}
\caption{$L_{ab}$ separates $\alpha_c$ and $\alpha_{i_k}$.\label{fig:Lseparates}}
\end{figure}

\begin{figure}
\subcaptionbox{Circuit type (A).}{
\begin{tikzpicture}
\def\x{2}
\def\y{2}
\def\shift{20}
% marked
\coordinate[shape=circle,inner sep=1pt,fill,label={[xshift=-8 pt,yshift=-6 pt]$\alpha_a$}] (P1) at (\x,\y+2);
\coordinate[shape=circle,inner sep=1pt,fill,label={[xshift=-8 pt,yshift=-6 pt]$\alpha_b$}] (P2) at (\x,\y+1);
\coordinate[shape=circle,inner sep=1pt,fill,label={[xshift=15 pt,yshift=-8 pt]$\alpha_{i_{m-3}}$}] (P3) at (\x,\y);
\coordinate[shape=circle,inner sep=1pt,fill,label={[xshift=8 pt,yshift=-6 pt]$\alpha_c$}] (P4) at (\x+1,\y+2);
\coordinate[shape=circle,inner sep=1pt,fill,label={[xshift=10 pt,yshift=-8 pt]$\alpha_{i_k}$}] (P5) at (\x+1,\y+1);
\node at (\x+0.4,\y+2.8) {$L_{ab}$};
\node at (\x+2,\y) {$\hspace{1pt}$};
% lines
\path[draw,dashed] (\x,\y-0.75) -- (\x,\y+3);
\path[draw,black] (P1) -- (P4) -- (P5) -- (P2) -- (P1);
\end{tikzpicture}
}
\hspace{10pt}
\subcaptionbox{Circuit type (B).}{
\begin{tikzpicture}
\def\x{2}
\def\y{2}
\def\shift{20}
% marked
\coordinate[shape=circle,inner sep=1pt,fill,label={[xshift=-8 pt,yshift=-6 pt]$\alpha_a$}] (P1) at (\x,\y+2);
\coordinate[shape=circle,inner sep=1pt,fill,label={[xshift=-8 pt,yshift=-6 pt]$\alpha_b$}] (P2) at (\x,\y+1);
\coordinate[shape=circle,inner sep=1pt,fill,label={[xshift=15 pt,yshift=-8 pt]$\alpha_{i_{m-3}}$}] (P3) at (\x,\y);
\coordinate[shape=circle,inner sep=1pt,fill,label={[xshift=8 pt,yshift=-6 pt]$\alpha_c$}] (P4) at (\x+0.5,\y+1.25);
\coordinate[shape=circle,inner sep=1pt,fill,label={[xshift=10 pt,yshift=-8 pt]$\alpha_{i_k}$}] (P5) at (\x+2,\y+1);
\node at (\x+0.4,\y+2.8) {$L_{ab}$};
% lines
\path[draw,dashed] (\x,\y-0.75) -- (\x,\y+3);
\path[draw,black] (P1) -- (P5) -- (P2) -- (P1);
\end{tikzpicture}
}
\hspace{10pt}
\subcaptionbox{Circuit type (C).}{
\begin{tikzpicture}
\def\x{2}
\def\y{2}
\def\shift{20}
% marked
\coordinate[shape=circle,inner sep=1pt,fill,label={[xshift=-8 pt,yshift=-6 pt]$\alpha_a$}] (P1) at (\x,\y+2);
\coordinate[shape=circle,inner sep=1pt,fill,label={[xshift=-8 pt,yshift=-6 pt]$\alpha_b$}] (P2) at (\x,\y+1);
\coordinate[shape=circle,inner sep=1pt,fill,label={[xshift=15 pt,yshift=-8 pt]$\alpha_{i_{m-3}}$}] (P3) at (\x,\y);
\coordinate[shape=circle,inner sep=1pt,fill,label={[xshift=8 pt,yshift=-6 pt]$\alpha_c$}] (P4) at (\x+2,\y+1);
\coordinate[shape=circle,inner sep=1pt,fill,label={[xshift=8 pt,yshift=-3 pt]$\alpha_{i_k}$}] (P5) at (\x+1,\y+1);
\node at (\x+0.4,\y+2.8) {$L_{ab}$};
% lines
\path[draw,dashed] (\x,\y-0.75) -- (\x,\y+3);
\path[draw,black] (P4) -- (P5) -- (P2);
\end{tikzpicture}
}
\caption{$L_{ab}$ does not separate $\alpha_c$ and $\alpha_{i_k}$.\label{fig:Ldoesnotseparate}}
\end{figure}
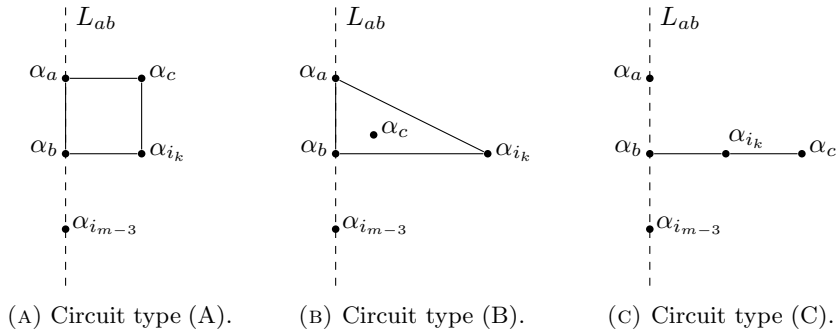

\begin{remark}
The real tropicalization of the family of real Laurent polynomials with support $\support$, providing real plane tropical curves with a singularity in $\ones{2}$ dual to a circuit of type (A) or (B), multiplied with the lineality group equals the set of equivalence classes of $\signedsecondaryfan{\support}$ corresponding to signed regular marked subdivisions containing a circuit of type (A) or (B) (cf. \cite[Lemma 4.2.4.26]{juergens}). However, for circuits of type (C) there are examples of equivalence classes which we only obtain partially in the product $\realtrop(\Realsingsat{\ones{2}})\odot_{\R} G\times\rowspace(A')$ (cf. \Cref{rema:realtropicaldiscriminant}).
For details see \cite[Section 4.2.4]{juergens}, in particular \cite[Example 4.2.4.32]{juergens}.
\end{remark}

\subsection{Classification of Real Plane Tropical Curves of Maximal Dimensional Type with a Singularity fixed in $\ones{2}$}\label{subsec:firststepsrealclassification}

In this section we exploit signed regular marked subdivisions arising from weight classes of $\realtrop(\Realsingsat{\ones{2}})$. Thereby, we focus on the $(+,+)$-chart as we fix the singularity in $p=\ones{2}$. In particular, for a given real tropical Laurent polynomial $f\in\RealTroplauringtwovars{}$ such that $\zeros{2}^+\in\RealTh{f}$ is a singularity we classify the local picture around $\zeros{2}^+\in C_{f,(+,+)}$. We stick to notations defined in \Cref{rema:planarcircuits} and \Cref{rema:classificationflagsunorientedmatroid}. Before giving the classification, note that $\signedsecondaryfan{\support}$ is not the parameter space of real plane tropical curves (\cite{markwigmarkwigshustin}). Many real tropical Laurent polynomials provide the identical real plane tropical curve.\smallskip

Consider a signed regular marked subdivision $T=\{(P_i,Q_i,s_{Q_i}):i=1,\ldots,k\}$. Recall that the type of $T$ is $\mathcal{H}=\{P_i:i=1,\ldots,k\}$ (cf. \Cref{defi:signedmarkedsubdivision}). In the complex case the type $\mathcal{H}$ is dual to a tropical curve $C'$ and we call $\mathcal{H}$ the type of $C'$. Let $C$ denote the real tropical curve dual to $T$. The chart $C_{v}$ of $C$ is dual to $T_v$ (cf. \Cref{rema:duality}) and, moreover, $|T_v|=|T|$, i.e. the type of $T_v$ equals the type of $T$. However, in contrast to the complex case, a chart $C_v$ of the real plane tropical curve $C$ dual to $T_v$ contains only parts of the curve $C'$ dual to $|T|$ owed to the sign conditions. The curves $C'$ dual to a given type $\mathcal{H}$ can be parametrized by an unbounded polyhedron in $\R^{b+2}$ where $b$ denotes the number of bounded edges in $C'$ and the two additional dimensions correspond to the spacial translation of the curve $C'$ in $\R^2$. The lengths of the bounded edges cannot be changed individually if the tropical curve $C'$ is of genus $g\geq 1$. Then, we have $2g$ not necessarily independent conditions corresponding to the closed circles in $C'$. The dimension of the parametrizing cone in $\R^{b+2}$ is called \textit{type-dimension of $\mathcal{H}$} and it is denoted by $\tdim{\mathcal{H}}$ (\cite[Subsection 2.2]{markwigmarkwigshustin}). As the real tropical curve $C$ inherits its polyhedra from $C'$ this cone also parametrizes real tropical curves $C$ of type $\mathcal{H}$.
\begin{lemma}[{\cite[Lemma 2.5]{markwigmarkwigshustin}}]\label{lemm:typedimension}
Let $T$ be a signed marked subdivision of $\Delta=\conv(\support)$ of type $\mathcal{H}$. Let $s\times\sigma_T\subset\signedsecondaryfan{\support}\Mod{m}$ denote the equivalence class of the secondary fan. Then,
\begin{equation}
\tdim{\mathcal{H}}\leq\dim(s\times\sigma_T)
\end{equation}
and we have equality if and only if all lattice points of $\Delta$ are marked points in $T$.
\end{lemma}
As a consequence, we can study real plane tropical curves of maximal dimensional type by restricting to signed regular marked subdivisions without white points.
\begin{definition}[Maximal dimensional type]\label{defi:maxtype}
An equivalence class $s\times\sigma_T\subset\signedsecondaryfan{\support}$ is of \textit{maximal dimensional type}\index{type-dimension!maximal} if all points of $T$ are marked, i.e. there are no white points in the subdivision $T$.
\end{definition}

% THEOREM: Classification max. dim. type codimension one
\begin{theorem}[Classification of singular real plane tropical curves of maximal dimensional type]\label{theo:classificationcurvesmaximaldimensionaltype}
Let $s\times\sigma_T\subset\signedsecondaryfan{\support}$ be an equivalence class maximal dimensional type. Consider the real tropical Laurent polynomial $f=\bigoplus_i(s_i,u_i)w^{\alpha_i}\in\RealTroplauringtwovars{}$ with support $\support$ defined by $(s,u)\in s\times\sigma_T$. Assume that $\RealTh{f}$ has a singularity at $\zeros{2}^+$. Then, the local picture of $\RealTh{f}_{(+,+)}$ around the singularity $\zeros{2}^+$ matches with one of the following cases:
\begin{enumerate}
\item $\zeros{2}^+$ is a $4$-valent vertex incident to $4$ edges of weight one,
\item $\zeros{2}^+$ is an isolated vertex of multiplicity 3,
\item $\zeros{2}^+$ is the midpoint of an edge of weight 2 that is connecting
\begin{itemize}
\item either two $1$-valent vertices, or
\item two $3$-valent vertices,
\end{itemize}
\item $\zeros{2}^+$ is contained in an interval of an edge of weight 2 connecting a 3-valent vertex and
\begin{itemize}
\item either a $1$-valent vertex, or
\item a $3$-valent vertex,
\end{itemize}
and the boundaries of the interval are the $3$-valent vertex and the midpoint, or
\item $\zeros{2}^+$ is contained in an infinite edge of weight 2 whose endpoint is a $3$-valent vertex.
\end{enumerate}
\end{theorem}
A classification of regular marked subdivisions arising from flags of flats of $\underline{M}$ can be found in \cite{markwigmarkwigshustin}. We exploit the known classification of subdivisions with respect to the additional sign conditions.
% PROOF
\begin{proof}
As $\zeros{2}^{+}\in\RealTh{f}$ is a singularity, we have $(s,u)\in\realtrop(\Realsingsat{\ones{2}})$. Hence, there is a real Laurent polynomial $F=\sum_i a_ix^{\alpha_i}$ with support $\support$ having a singularity at $\ones{2}$ and $\realtrop(a)=(s,u)$.
We have $\realtrop(\Realsingsat{\ones{2}})=\bigcap_{C\in\circuits}\RealTh{l_{C}}$ where $\circuits$ denotes the set of signed circuits of the oriented matroid $\idealmatroid{\Ideal{I}}$ (\Cref{theo:realtropicalbasis}). Moreover, \Cref{theo:realtropicallinearspace=signedbergmanfan} determines its $s$-charts, i.e. we have \[\left(\realtrop(\Realsingsat{\ones{2}})\right)_{s}=\signedbergmanfan{s}{M}.\] As $u\in\signedbergmanfan{s}{M}$ there is a maximal $s$-flag $\flag\triangleleft M$ such that $u\in\sigma_\flag\subset\signedbergmanfan{s}{M}$. Note that the real tropical Laurent polynomial defined by $(s,u)$ equals $f$, i.e. $C_{f,(+,+)}$ is dual to $T_{(+,+)}$ where $T$ denotes the regular marked subdivision obtained from $u$. Now, one of the following cases applies to $\flag$:
% FALLUNTERSCHEIDUNG
\begin{enumerate}
% flag type (a)
\item The flag $\flag$ is an $s$-flag of type (a), i.e. $F_{m-3,m-4}=\{a,b,c,d\}$ corresponds to a circuit of type (A) or (B) in $T$ (cf. \Cref{lemm:typeaflags}). Thus, we have $w_{a}=w_{b}=w_{c}=w_{d}\geq w_{k}$ for all $k\neq a,b,c,d$. In order to have a subdivision of maximal dimensional type, the circuit corresponding to $F_{m-3,m-4}$ cannot contain any other points.
\begin{enumerate}
\item If $F_{m-3,m-4}=\{a,b,c,d\}$ provides a signed circuit of type (A) then $\conv(\alpha_a,\alpha_b,\alpha_c,\alpha_d)$ is a quadrangle in the signed regular marked subdivision $T$. The signs at vertices coincide with the signed circuit (cf. \Cref{lemm:typeaflags}), i.e. they are alternating (cf. \Cref{fig:quadrangle}). Hence, any edge of the quadrangle is dual to an (maybe unbounded) edge in $C_{f,(+,+)}$. The quadrangle itself is dual to a vertex in $C_{f,(+,+)}$. There, $|f|$ attains the maximum at all four terms corresponding to the vertices. Since the coefficients of all terms have equal modulus, the vertex is at $\zeros{2}\in\RealTh{f}_{(+,+)}$. Hence, quadrangle is dual to a $4$-valent vertex at $\zeros{2}\in C_{f,(+,+)}$.
\item If $F_{m-3,m-4}=\{a,b,c,d\}$ is a signed circuit of type (B) then $\conv(\alpha_a,\alpha_b,\alpha_c,\alpha_d)$ is a triangle with an interior point in the signed regular marked subdivision $T$. The signs at vertices coincide with the signs of $C$ (cf. \Cref{lemm:typeaflags}), i.e. the interior point obtains (w.l.o.g.) ``$-$'' and the exterior points forming the triangle get ``$+$'' (cf. \Cref{fig:triangle}). The triangle is dual to a vertex in $C_{f,(+,+)}$ and we can solve for its coordinates. The maximum of $|f|$ obtained from the real tropical polynomial $f$ attains its maximum at all four terms as the moduli are equal. Hence, the vertex equals $\zeros{2}$. Moreover, the sign at the interior point differs from the signs of exterior points, i.e. $\zeros{2}\in C_{f,(+,+)}$ corresponds to the triangle. However, the edges forming the boundary of $\conv(\alpha_a,\alpha_b,\alpha_c,\alpha_d)$ are incident to vertices with equal signs. Consequently, $\zeros{2}\in C_{f,(+,+)}$ is an isolated point as the edges of the triangle have no dual counterpart in $C_{f,(+,+)}$. The isolated vertex $\zeros{2}^{+}\in\RealTh{f}$ differs from other vertices by its multiplicity, which is the area of the dual circuit $C=\conv(\alpha_{a},\alpha_{b},\alpha_{c},\alpha_{d})$. 
\end{enumerate}
% flag type (b)
\item The flag $\flag$ is an $s$-flag of type (b) and $F_{m-3,m-4}=\{a,b,c\}$ is a circuit of type (C) (cf. \Cref{lemm:typebflags}). Thus, $w_a=w_b=w_c\geq w_j$ for all $j\in F_l$ with $l\leq m-4$ and the points $\alpha_a,\alpha_b$ and $\alpha_c$ get highest and equal height. Consequently, we see a circuit $\conv(\alpha_a,\alpha_b,\alpha_c)$ of type (C) (cf. \Cref{fig:line}) in the signed regular marked subdivision $T=T_{(+,+)}$ formed by the points $\alpha_a,\alpha_b$ and $\alpha_c$. In particular, the signs at its vertices are alternating (cf. \Cref{lemm:typebflags}), i.e. $C$ is dual to an edge $e\in\RealTh{f}_{(+,+)}$. Note that $e$ passes through $\zeros{2}$. To see this, note that the coefficients of the monomials of $f$ corresponding to $\alpha_{a}$, $\alpha_{b}$ and $\alpha_{c}$ have equal modulus. Thus, $|f|$ attains the maximum at $w$ exactly at these three monomials (see inequality above). Hence, $e\subset\{y=0\}$. The points $\alpha_d$ and $\alpha_e$ also get equal height. Moreover, these points get highest height of all points that are not on the line $L_{abc}$ through $\alpha_a,\alpha_b$ and $\alpha_c$. We do not allow any white points in $T$ and, therefore, the points $\alpha_d$ and $\alpha_e$ have to be at minimal distance to the signed circuit $\conv(\alpha_a,\alpha_b,\alpha_c)$. Otherwise, the polygon containing $C$ and (w.lo.g.) $\alpha_{d}$, which is at non-minimal distance, would contain other points beside $\alpha_{a},\alpha_{b},\alpha_{c},\alpha_{d}$ that, in this particular case, must have higher height as $\alpha_{d}$. This cannot be true as $\flag$ is a flag of type (b) (cf. \Cref{rema:classificationflagsunorientedmatroid}). For more details, see \cite[Section 3.3]{markwigmarkwigshustin}. The points $\alpha_{d}$ and $\alpha_{e}$ have to be vertices of signed marked polytopes of $T$ in order to get a maximal dimensional type of smallest codimension. Let $L_{abc}$ denote the affine line through $\conv(\alpha_a,\alpha_b,\alpha_c)$. Two subcases appear:
\begin{enumerate}
\item Suppose $L_{abc}$ separates $\alpha_d$ and $\alpha_e$. The vertices of $\conv(\alpha_a,\alpha_b,\alpha_c)$ have highest but equal height $\mu$ and $\alpha_d$ and $\alpha_e$ get highest height $\lambda$ of all points not on $L_{abc}$, i.e. $\mu\geq \lambda$. Hence, $\conv(\alpha_a,\alpha_b,\alpha_c,\alpha_d)$ and $\conv(\alpha_a,\alpha_b,\alpha_c,\alpha_e)$ form the triangles in $T$ that contain $C$ as we do not allow white points. Assume that $L_{abc}=\{x=1\}$ and $\alpha_d=\zeros{2}$. We solve for the coordinates of the vertex dual to the triangle:
\begin{align*}
\lambda& + \sk{\alpha_d}{(x,y)} = \mu +\sk{\alpha_a}{(x,y)} =\mu + \sk{\alpha_b}{(x,y)}\\
\Leftrightarrow\quad\lambda& = \mu + x + (\alpha_a)_2 y = \mu + x + (\alpha_b)_2 y.
\end{align*}
Without restriction, $\lambda=0$ and, therefore, $\mu\geq 0$. Then, $y=0$ and $x=-\mu$, i.e. the vertex is at $(-\mu,0)$. Using similar arguments and symmetry, and the fact that $\alpha_d$ and $\alpha_e$ have equal height, the second vertex is at $(\mu,0)$. We observe that the distance of $(0,0)$ to the vertices is equal. These vertices are part of $\RealTh{f}_{(+,+)}$ as the polygons $\conv(\alpha_a,\alpha_b,\alpha_c,\alpha_d)$ and $\conv(\alpha_a,\alpha_b,\alpha_c,\alpha_e)$ contain the signed circuit that has alternating signs. Whether the vertices dual to $\conv(\alpha_a,\alpha_b,\alpha_c,\alpha_d)$ and $\conv(\alpha_a,\alpha_b,\alpha_c,\alpha_e)$ are incident to other edges purely depends on the signs at $\alpha_{d}$ and $\alpha_{e}$. Note that the vertex dual to $\conv(\alpha_a,\alpha_b,\alpha_c,\alpha_d)$ is incident to two more edges if and only if $s_a\neq s_d$ (without restriction, we assume that $\alpha_a$ is a boundary vertex of the signed circuit $\conv(\alpha_a,\alpha_b,\alpha_c)$). As $\flag$ is an $s$-flag of type (b) we know that $s_d=s_e$ (cf. \Cref{lemm:typebflags}), i.e. either both vertices dual to $\conv(\alpha_a,\alpha_b,\alpha_c,\alpha_d)$ and $\conv(\alpha_a,\alpha_b,\alpha_c,\alpha_e)$ are 3-valent or both vertices are 1-valent.
\item If $L_{abc}$ does not separate $\alpha_d$ and $\alpha_e$ then they are on the same side of $L_{abc}$. Recall that they need to be at minimal distance. Hence, they are on a line parallel to $L_{abc}$. Thus, $\conv(\alpha_a,\alpha_b,\alpha_c,\alpha_d,\alpha_e)$ is a quadrangle not covering any other points. Let $\mu$ and $\lambda$ denote the heights of $\alpha_a,\alpha_b,\alpha_c$ and $\alpha_d,\alpha_e$ respectively. If $\conv(\alpha_a,\alpha_b,\alpha_c)$ is not contained in the boundary of $T$, there must be another vertex $\alpha_f$ with height $\nu$ at minimal distance on the other side of $L_{abc}$ that forms a triangle with $\alpha_a,\alpha_b,\alpha_c$. We solve for the coordinates of the vertex dual to the triangle as above, i.e. we suppose that $L_{abc}=\{x\in\R^2:x_1=1\}$, $\alpha_f=\zeros{2}$. Then:
\begin{align*}
\nu& + \sk{\alpha_f}{(x,y)} = \mu + \sk{\alpha_a}{(x,y)} = \mu +\sk{\alpha_b}{(x,y)}\\
\Leftrightarrow\quad\nu& = \mu + x + (\alpha_a)_2 y = \mu +x+(\alpha_b)_2 y.
\end{align*}
By assumption we have $\nu<\lambda<\mu$ and we can assume that $\nu = 0$. Thus, the vertex dual to the triangle $\conv(\alpha_a,\alpha_b,\alpha_c,\alpha_f)$ is at $(-\mu,0)$. However, the vertex dual to the quadrangle $\conv(\alpha_a,\alpha_b,\alpha_c,\alpha_d,\alpha_e)$ is at $(\mu-\lambda, 0)$, due to its describing equations:
\begin{align*}
\lambda& + \sk{\alpha_d}{(x,y)} = \lambda + \sk{\alpha_e}{(x,y)} = \mu +\sk{\alpha_a}{(x,y)} = \mu +\sk{\alpha_b}{(x,y)}\\
\Leftrightarrow\quad\lambda& + 2x + (\alpha_d)_2 y = \lambda +2x+(\alpha_e)_2 y =\mu + x + (\alpha_a)_2 y =\mu + x + (\alpha_b)_2 y.
\end{align*}
Note that the distance from the vertex at $(-\mu,0)$ dual to the triangle to the singularity at $(0,0)\in C_{(+,+)}$ is bigger than the distance from the distance of the vertex at $(\mu-\lambda,0)$ dual to the quadrangle. In any case, the vertex dual to $\conv(\alpha_a,\alpha_b,\alpha_c,\alpha_d,\alpha_e)$ is $3$-valent as $\alpha_d$ and $\alpha_e$ are on the same side of $L_{abc}$, i.e. $s_d\neq s_e$ (cf. \Cref{lemm:typebflags}). Thus, one edge of $\conv(\alpha_a,\alpha_b,\alpha_c,\alpha_d,\alpha_e)$ has vertices with equal signs, i.e. the quadrangle is dual to a $3$-valent vertex. Whether the vertex dual to the triangle $\conv(\alpha_a,\alpha_b,\alpha_c,\alpha_f)$ has higher valence purely depends on the sign $s_f$: if (w.l.o.g.) $\alpha_a$ is a vertex of $\conv(\alpha_a,\alpha_b,\alpha_c)$ then the vertex dual to the triangle is $3$-valent if and only if $s_a\neq s_f$. If $s_a=s_f$ then the vertex dual to the triangle is $1$-valent. If $\conv(\alpha_a,\alpha_b,\alpha_c)$ is contained in the boundary we see the quadrangle $\conv(\alpha_a,\alpha_b,\alpha_c,\alpha_d,\alpha_e)$ in $T$ dual to a vertex at $(\mu-\lambda,0)\in C_{(+,+)}$ and the singularity lies on the infinite edge dual to the circuit $\conv(\alpha_a,\alpha_b,\alpha_c)$.
\end{enumerate}
\end{enumerate}
\end{proof}